\documentclass{scrartcl}
\usepackage[T1]{fontenc}
\usepackage[utf8]{inputenc}
\usepackage{amsmath, amsthm, amssymb}
\usepackage{dsfont}
\usepackage{graphicx}
\usepackage{color}
\usepackage[format=plain]{caption}%ben?tigt, um mit dem Befehl \caption*{} eine Bildunterschrift zu setzen, die nicht mit Abbildung X eingeleitet wird. 'normal'=\setcapindent{0}
\usepackage{capt-of}%ben?tigt, um 'figure' und 'table' nebeneinander zu setzen, siehe 'http://www.faqs.org/faqs/de-tex-faq/part6/' Punkt 6.1.7, aufgerufen am 01.08.11.
\usepackage{subcaption}
\usepackage{multicol}
\usepackage{bm}
\usepackage{dlfltxbcodetips} %gebraucht fuer grosses X fuer kartesische Produkte

\newcommand{\norm}[1]{\left\| #1 \right\|}  %Norm
\newcommand{\scprd}[1]{\left\langle #1 \right\rangle}  %Skalarprodukt
\renewcommand{\d}{\,\mathrm{d}} %Differenzial in Integralen
\newcommand{\N}{\mathbb{N}}  %Zahlbereiche
\newcommand{\Z}{\mathbb{Z}}
\newcommand{\Q}{\mathbb{Q}}
\newcommand{\R}{\mathbb{R}}

\newcommand{\supp}{\operatorname{supp}}

\DeclareMathOperator*{\esssup}{ess-sup}  %den Stern brauche ich, damit unter den Operator Grenzen geschrieben werden koennen (wie bei \lim oder \sup)
\DeclareMathOperator*{\essinf}{ess-inf}
\DeclareMathOperator*{\argmin}{argmin}
\newcommand{\rint}{\operatorname{r-int}}
\newcommand{\Av}{\operatorname{Av}}
\newcommand{\eps}{\varepsilon}
\renewcommand{\phi}{\varphi}
\newcommand{\ul}{\underline}
\newcommand{\ol}{\overline}

\numberwithin{equation}{section}
\newtheorem{thm}{Theorem}[section]
\newtheorem{cor}[thm]{Corollary}
\newtheorem{prop}[thm]{Proposition}
\newtheorem{lm}[thm]{Lemma}

\theoremstyle{definition} 
\theoremstyle{definition}

\title{Approximation, characterization, and continuity of multivariate monotonic regression functions}

\author{Jochen Schmid\\  
\small Fraunhofer Institute for Industrial Mathematics (ITWM), 67663 Kaiserslautern, Germany\\ 
\small jochen.schmid@itwm.fraunhofer.de}  

\date{}

\begin{document}

\maketitle

\begin{abstract}
\small{ \noindent 
We deal with monotonic regression of multivariate functions $f: Q \to \R$ on a compact rectangular domain $Q$ in $\R^d$, where monotonicity is understood in a generalized sense: as isotonicity in some coordinate directions and antitonicity in some other coordinate directions. As usual, the monotonic regression of a given function $f$ is the monotonic function $f^*$ that %among all monotonic functions 
has the smallest (weighted) mean-squared distance from $f$.
We establish %rigorously establish/develop
a simple general approach to %approximately 
compute monotonic regression functions: namely, we show that the monotonic regression $f^*$ of a given function $f$ can be approximated arbitrarily well -- with simple bounds on the approximation error in both the $2$-norm and the $\infty$-norm -- by the monotonic regression $f_n^*$ of grid-constant functions $f_n$. %and that $f_n^*$, in turn, can be computed using any of the well-known discrete monotonic regression algorithms. 
We also establish the continuity of the monotonic regression $f^*$ of a continuous function $f$ %with continuous weight $w$
along with an explicit averaging formula for $f^*$. 
And finally, we deal with generalized monotonic regression where the mean-squared distance from standard monotonic regression is replaced by more complex distance measures which arise, for instance, in maximum smoothed likelihood estimation. We will see that the solution of such generalized monotonic regression problems is simply given by the standard monotonic regression $f^*$. %coincides with the    
}
\end{abstract}

{ \small \noindent 
%\emph{AMS Subject Classification (2010):} 
%\\
}

\section{Introduction}

When trying to learn an unknown functional relationship $Q \ni x \mapsto y(x) \in \R$ on some rectangular (cuboid) domain $Q$ in $\R^d$, one can often rely not only on experimental data but also on theoretical prior knowledge. %one can often rely on theoretical prior knowledge in addition to experimental data. An important example of such prior knowledge ... 
Such prior knowledge can consist, for instance, in expert knowledge about the monotonicity behavior of the unknown function $y$. It could be known, for instance, that $y$ is isotonic in some (possibly all or no) coordinate directions and antitonic in some other (possibly all or no) coordinate directions. %A bit more precisely, %In other words,
In slightly more precise and concise terms, it could be known that $y$ is monotonic with monotonicity signature $\sigma \in \{-1,0,1\}^d$, %(or $\sigma$-monotonic, for short), 
where $\sigma_i = \pm 1$ %for a given $i$
means that $y$ is isotonic or antitonic in the $i$th coordinate direction, respectively, whereas $\sigma_i =0$ means that nothing is known about the monotonicity of $y$ in the $i$th coordinate direction. %$y$ is of unknown monotonicity in the $i$th coordinate direction.
We will %also call such functions $\sigma$-monotonic for brevity. 
refer to such functions as $\sigma$-monotonic, for brevity.  
\smallskip

Incorporating such general monotonicity knowledge can be achieved, for instance, by first training an initial model $f$ based only on experimental data %purely data-based model $f$
and then monotonizing this initial model appropriately, that is, to replace it by an appropriate $\sigma$-monotonic model. 
A prominent way of monotonization is by rearrangement: here, one replaces the initial model $f$ by a suitably defined 
$\sigma$-monotonic rearrangement $r_{\sigma}(f)$ of $f$. See, for instance,~\cite{DeNePi06}, \cite{DeSc06}, \cite{ChFeGa09} for this rearrangement-based approach to monotonization.
Another even more prominent way of monotonization is by projection: here, one replaces the initial square-integrable model $f$ by the projection $p_{\sigma}(f)$ of $f$ onto the set 
\begin{align}
L^2_{\sigma}(Q,w\cdot\lambda, \R) := \{ g \in L^2(Q,w\cdot\lambda,\R): g \text{ has a $\sigma$-monotonic representative} \}
\end{align}
where $w$ is a weight function bounded above and below by positive finite constants and $\lambda$ is Lebesgue measure on $Q$. In other words, one replaces $f$ by the $\sigma$-monotonic square-integrable function $p_{\sigma}(f) = p_{\sigma}^w(f)$ that has the smallest weighted mean-squared distance from $f$, that is, %or, for short, 
\begin{align} \label{eq:def-p_sigma(f)-intro}
p_{\sigma}^w(f) = \argmin_{g \in L^2_{\sigma}(Q,w\cdot\lambda,\R)} \int_Q |f(x)-g(x)|^2 w(x)\d x.
\end{align}  
Commonly, the projection $p_{\sigma}^w(f)$ is also called the $\sigma$-monotonic regression (function) of $f$ with weight $w$ and, by extension, the underlying minimization problem for~\eqref{eq:def-p_sigma(f)-intro}  is referred to as monotonic regression (problem). See, for instance, \cite{BaBaBrBr}, \cite{RoWrDy}, \cite{Ma91}, \cite{MaMaTuWa01}, \cite{LiDu14} for this projection-based approach to monotonization. 
Alternative ways of incorporating monotonicity knowledge in models of various complexities are developed in~\cite{HaHu01}, \cite{RiVe10}, \cite{GuCo16}, and the references from Table~1 of~\cite{GuCo16}, for instance.
\smallskip

In this paper, we will be exclusively dealing with the monotonic regression of multivariate functions $f: Q \to \R$ on some compact rectangular (cuboid) domain %$Q$ in $\R^d$. 
\begin{align} \label{eq:Q}
Q := [a_1,b_1] \times \dotsb \times [a_d,b_d] %\subset \R^d
\end{align}
in $\R^d$
with $a_i < b_i$.  
%we will be dealing with the monotonization of multivariate functions $f: Q \to \R$ by projection. %we will be taking the projection-based approach to monotonization of multivariate functions $f: Q \to \R$. 
We will develop a simple general approach to compute monotonic regression functions $p_{\sigma}^w(f)$ and we will establish theoretical results on the characterization and continuity of $p_{\sigma}^w(f)$. As can be expected from the discrete case, treating multivariate functions is considerably more complex than treating univariate functions. 
In detail, the contents and contributions of the present paper can be described as follows.
\smallskip

Section~2 provides some general facts and definitions that will be used again and again throughout the text. 
In Section~3, we will show that the monotonic regression $p_{\sigma}^w(f)$  for grid-constant functions $f$ and $w$ is grid-constant too and determined %given 
by the corresponding discrete monotonic regression 
\begin{align} \label{eq:discrete-monotonization}
\argmin_{g \in \ell^2_{\sigma}(G,\R)} \sum_{x\in G} |f(x)-g(x)|^2 w(x),
\end{align}
where $\ell^2_{\sigma}(G,\R) := \{g \in \ell^2(G,\R): g \text{ is $\sigma$-monotonic}\}$ and $G$ is the grid on the cells of which %on whose cells 
$f$ and $w$ are constant. 
In order to compute~\eqref{eq:discrete-monotonization}, one can use any of the well-known discrete algorithms from the literature~\cite{BaBaBrBr} (Section~2.3), \cite{RoWrDy} (Section~1.4), \cite{MaMu85}, \cite{QiEd96}, \cite{HoQu03}, \cite{SpWaWi03}, \cite{St13}, \cite{St15}, \cite{KyRaSa15}. See, in particular, Table~1 of~\cite{St13} and of~\cite{KyRaSa15} as well as~\cite{St19} for a nice overview and a comparison of the computational complexity of various of these algorithms. 
\smallskip

In Section~4, we establish two approximation results saying that the monotonic regression $p_{\sigma}^w(f)$ can be approximated arbitarirly well %-- in the $2$-norm or the $\infty$-norm -- %depending on whether $f$ is only square-integrable or even essentially bounded --
by monotonic regression functions $p_{\sigma}^{w_n}(f_n)$ of suitable grid-constant functions $f_n$ and $w_n$. In case $f$ is merely square-integrable the approximation is w.r.t.~the $2$-norm, while it is w.r.t.~the $\infty$-norm in case $f$ is even essentially bounded. We also establish simple upper bounds on the approximation errors in both norms. Combining the results from Section~3 and~4, we obtain a simple general computational methodology that reduces the computation of monotonic regressions of functions $f$ and $w$ on the continuous (non-discrete) domain $Q$ to the computation of corresponding discrete monotonic regressions. Compared to the fairly different computational method from~\cite{LiDu14} our methodolgy has several advantages:
\begin{itemize}
\item it is simpler, both conceptionally and computationally
\item it is more informative, as it yields %contains 
information on the rate of convergence
\item it is more flexible, because for the computation of the discrete approximants one can choose any of the known discrete algorithms.   
\end{itemize} 

In Section~5, we deal with generalized monotonic regression where the mean-squared distance from standard monotonic regression~\eqref{eq:def-p_sigma(f)-intro} is replaced by more complex distance measures which arise, for instance, in maximum smoothed likelihood estimation. Specifically, generalized monotonic regression is the problem to find, for a given square-integrable function $f$ and a weight function $w$, the minimizer
\begin{align} \label{eq:generalized-mon-regr-intro}
\argmin_{g \in L^2_{\sigma}(Q,w\cdot\lambda,\R)} \int_Q \Delta_{\Phi}(f(x),g(x))w(x) \d x
\end{align}
with $\Delta_{\Phi}(u,v)$ being a generalized distance measure between $u,v \in \R$ based on some convex function $\Phi$ on $\R$. We will show that, in spite of the generally much more complex objective function in~\eqref{eq:generalized-mon-regr-intro}, the minimizer~\eqref{eq:generalized-mon-regr-intro} coincides with the minimizer $p_{\sigma}^w(f)$ of the standard monotonic regression problem~\eqref{eq:def-p_sigma(f)-intro}. In doing so, we generalize a well-known result from the discrete case~\cite{BaBaBrBr}, \cite{RoWrDy} as well as a result from the recent paper~\cite{GrJo10}. In that paper, the case of univariate and continuous functions $f$ and $w$ is considered, but the strategies of proof from~\cite{GrJo10} do not carry over to our case of multivariate functions. In our proof, we make essential use of our computational methodology from Section~3 and~4 (especially, the approximation result w.r.t.~the $2$-norm).
\smallskip

In Section~6, we finally deal with the special case of continuous functions $f$ and $w$. We establish the continuity of the monotonic regression function $p_{\sigma}^w(f)$ in that case -- along with an explicit averaging formula for $p_{\sigma}^w(f)$, which in the univariate special case reduces to the well-known formula
\begin{align}
p_{\sigma}^w(f)(x) = \inf_{u < x} \sup_{v > x} \bigg( \int_u^v f(t) w(t)\d t \bigg) \bigg/ \bigg( \int_u^v w(t)\d t \bigg) 
\qquad (x \in (a,b))
\end{align}
from~\cite{Ma91}, \cite{AnSo11}, \cite{MaMaTuWa01}. We thus generalize a result from~\cite{GrJo10} for univariate functions, but the multivariate case requires a completely different strategy of proof. In particular, in contrast to the univariate case, the continuity of $p_{\sigma}^w(f)$ can no longer -- at least not conveniently -- be %easily %conveniently 
inferred from the averaging formula in the multivariate case. In our proof, we make essential use of our computational methodology from Section~3 and~4 (especially, the approximation result w.r.t.~the $\infty$-norm along with a well-known discrete averaging formula). 
\smallskip

Apart from the specific notations and terminology explained in Section~2 and along the way in later sections, we will use the following general notational conventions. 
When speaking of monotonically increasing or decreasing functions, we will mean a monotonically non-decreasing or monotonically non-increasing function defined on a subset of $\R$. With $d$ we will always denote an arbitary non-negative integer and, unless explicitly stated otherwise, $\sigma$ will be a tuple from $\{-1,0,1\}^d$. With the symbol $Q$ we will always denote a cuboid set in $\R^d$ of the form~\eqref{eq:Q} and $\operatorname{int} E$ and $\rint F$ will denote the interior of a set $E\subset \R^d$ and the relative interior of a subset $F \subset Q$. Also, 
%$|\cdot|$ and $|\cdot|_{\infty}$ are the $2$-norm and the maximum norm on $\R^d$, respectively, while
%\begin{align}
%B_{\eps}(x) := \{ y \in \R^d: |y-x|_{\infty} < \eps \}.
%\end{align}
for $x \in \R^d$,
\begin{align}
B_{\eps}(x) := \{ y \in \R^d: |y-x|_{\infty} < \eps \}
\end{align}
is the $\eps$-ball around $x$ w.r.t.~the maximum norm $|\cdot|_{\infty}$ on $\R^d$. %where $|\cdot|_{\infty}$ is the maximum norm on $\R^d$. 
And finally, for a measurable subset $X$ of $\R^d$ %$\norm{\cdot}_q$, $\norm{\cdot}_{\infty}$ and $\norm{\cdot}_{\mathrm{sup}}$ denote, respectively, the $q$-norm, the essential supremum norm and the supremum norm on $X$,
and a function $f: X\to \R$, 
\begin{align}
\norm{f}_q := \bigg( \int_X |f(x)|^q \d x \bigg)^{1/q} \qquad (q \in [1,\infty))
\quad \text{and} \quad
\norm{f}_{\infty} := \esssup_{x \in X} |f(x)|,
\end{align}
%for $q \in [1,\infty)$, 
$\norm{f}_{\mathrm{sup}} := \sup_{x\in X} |f(x)|$, 
while $\chi_E$ stands for the characteristic function of a subset $E \subset X$, and $\mathcal{Q}_X$ denotes the Lebesgue sigma-algebra of $X$ (the more common symbol $\mathcal{L}$ is reserved for lower sets in this paper).

\section{Some preliminaries} %{Some general facts}

In this section, we collect the definitions and general facts we will need for our main results later on. We begin with the definition of isotonic and, in particular, $\sigma$-monotonic functions. 
Suppose $X$ is a set endowed with a partial order $\le$ (that is, a reflexive, antisymmetric, and transitive relation on $X$). A function $f:X \to \R$ is then called \emph{isotonic} iff %for all $x,y \in X$ with $x \le y$ one has $f(x) \le f(y)$. 
\begin{align} \label{eq:def-isotonic-fct}
\text{for all $x,y \in X$ with $x \le y$ one has $f(x) \le f(y)$.}
\end{align}
It is called \emph{antitonic} iff $-f$ is isotonic. 
We will occasionally need the following well-known and trivial %easily verified 
characterization of isotonic functions in terms of lower or upper sets. As usual, a subset $L$ or $U$ of $X$ is called a lower (upper) set on $X$ iff for all $x,y \in X$ with $x\le y$ the inclusion $y \in L$ ($x \in U$) implies $x \in L$ ($y \in U$). We will use the following short-hand notations: %write
\begin{align*}
\mathcal{L}_{\le}(X) := \{L: L \text{ is a lower set on } X\},
\qquad
\mathcal{U}_{\le}(X) := \{U: U \text{ is an upper set on } X\}.
\end{align*}

\begin{lm} \label{lm:char-isotonic-fcts}
Suppose $X$ is a set with a partial order $\le$. A function $f:X \to \R$ is then isotonic w.r.t.~$\le$ if and only if one of the following four conditions is satisfied:
\begin{multicols}{2} %\raggedcolumns
\begin{itemize}
\item[(i)] $f^{-1}((-\infty,c]) \in \mathcal{L}_{\le}(X)$ for all $c \in \R$
\item[(ii)] $f^{-1}((-\infty,c)) \in \mathcal{L}_{\le}(X)$ for all $c \in \R$
\item[(iii)] $f^{-1}([c,\infty)) \in \mathcal{U}_{\le}(X)$ for all $c \in \R$
\item[(iv)] $f^{-1}((c,\infty)) \in \mathcal{U}_{\le}(X)$ for all $c \in \R$.
\end{itemize}
\end{multicols}
\noindent Additionally, a characteristic function $\chi_U$ defined on $X$ is isotonic w.r.t.~$\le$ if and only if $U \in \mathcal{U}_{\le}(X)$. And finally, 
\begin{align}
\mathcal{U}_{\le}(X) = \{X\setminus L: L \in \mathcal{L}_{\le}(X) \}.
\end{align} 
\end{lm}

In this paper, we will be dealing a lot with multivariate functions on subsets of $\R^d$ that are monotonic -- isotonic or antitonic -- in some coordinate directions and non-monotonic -- or, more precisely, not known to be monotonic -- in some other coordinate directions. %monotonically increasing or decreasing in some coordinate directions and 
Such a general monotonicity behavior of multivariate functions can be expressed as isotonicity w.r.t.~a suitable partial order on $\R^d$. Specifically, we will call a function $f: X \to \R$ with $X \subset \R^d$ \emph{$\sigma$-monotonic with (monotonicity) signature $\sigma \in \{-1,0,1\}^d$} iff it is isotonic w.r.t.~the partial order $\le_{\sigma}$ on $X$ defined in the following way: %$x \le_{\sigma} y$
\begin{align} \label{eq:def-le_sigma}
x \le_{\sigma} y \qquad \text{iff} \qquad \sigma_i x_i \le \sigma_i y_i \qquad (i \in \hat{I}) \qquad \text{and} \qquad x_i = y_i \qquad (i \in \dot{I}),
\end{align} 
where $\hat{I} := \{i \in \{1,\dots,d\}: \sigma_i = \pm 1\}$ and $\dot{I} := \{ i \in \{1,\dots,d\}: \sigma_i = 0\}$. 
In all our results, $X$ will be a rectangular set in $\R^d$ minus, possibly, some null set. %If $X = S_1\times \dotsb S_d$ is genuinely rectangular (with no null set subtracted), then a function $f$ on $X$ is $\sigma$-monotonic if and only if 
As is easily verified, a multivariate function $f$ on a rectangular set $X = S_1 \times \dots \times S_d$ (with no null set subtracted) is $\sigma$-monotonic if and only if for every coordinate direction $i \in \{1,\dots,d\}$  the univariate functions
\begin{align*}
S_i \ni \xi \mapsto f(x_1, \dots, x_{i-1}, \xi, x_{i+1}, \dots, x_d)
\end{align*}
are $\sigma_i$-monotonic for all fixed $x_j \in S_j$ with $j \ne i$, where a $(\pm 1)$-monotonic univariate function is just a monotonically increasing or decreasing function, respectively, and a $0$-monotonic univariate function is an arbitrary function (no monotonicity imposed). %where $(\pm 1)$-monotonicity is nothing but isotonicity or antitonicity, respectively, and $0$-monotonicity imposes no monotonicity constraint.  

\subsection{Case of general measure-space domains}%{Case of general measure spaces}

We now consider the case where the set $X$, in addition to a partial order $\le$, is endowed with a complete measure $\mu$. In this case, we call a function $f:X \to \R$ \emph{essentially isotonic w.r.t.~$\le$} iff there is a $\mu$-null set $N$ such that the restriction $f|_{X\setminus N}$ is isotonic w.r.t.~$\le$ (restricted to $X\setminus N$).
It is clear that if $f$ is essentially isotonic, then so is every function $f'$ that coincides with $f$ $\mu$-almost everywhere. %is $\mu$-almost everywhere equal to $f$. 

\begin{lm} \label{lm:L^q_sigma-convex}
Suppose $(X,\mathcal{A},\mu)$ is a complete measure space with a partial order $\le$ and $I\subset \R$ is a closed interval. Then
\begin{align} \label{eq:def-isotonic-el-of-L^q}
L_{\le}^q(X,\mu,I) := \big\{ f \in L^q(X,\mu,I): &\text{ some (hence every) representative of } f \notag \\
&\text{ is essentially isotonic w.r.t.} \le \big\}
%L_{\le}^q(X,\mu,I) := \big\{ f \in L^q(X,\mu,I): f \text{ has an essentially $\le$-isotonic representative } f_0 \big\}
\end{align}
is a closed convex subset of $L^q(X,\mu,\R)$ for every $q \in [1,\infty)$. 
\end{lm}

\begin{proof}
In view of the convexity of intervals, it is clear that $L_{\le}^q(X,\mu,I)$ is a convex subset of $L^q(X,\mu,\R)$ and it remains to prove the closedness of $L_{\le}^q(X,\mu,I)$. So, let $f_n \in L_{\le}^q(X,\mu,I)$ and $f \in L^q(X,\mu,\R)$ with $f_n \longrightarrow f$ w.r.t.~$\norm{\cdot}_{q,\mu}$ and let $f_0, f_{n0}$ be arbitrary representatives of $f$, $f_n$. Then there is a subsequence $(n_k)$ such that 
\begin{align} \label{eq:L^q_sigma-convex-1}
f_{n_k 0}(x) \longrightarrow f_0(x) \qquad (k\to\infty)
\end{align}  
for $\mu$-almost every $x \in X$. Since the functions $f_{n_k 0}$ are all essentially $\le$-isotonic, it follows from~\eqref{eq:L^q_sigma-convex-1} that $f_0$ is essentially $\le$-isotonic as well. Since $f_{n_k 0}(x) \in I$ for $\mu$-almost every $x \in X$ and since $I$ is closed, it further follows from~\eqref{eq:L^q_sigma-convex-1} that $f_0(x) \in I$ for $\mu$-almost every $x \in X$. So, $f_0$ is an essentially $\le$-isotonic representative of $f$ and $f \in L^q(X,\mu,I)$. In other words, $f \in L_{\le}^q(X,\mu,I)$, as desired. 
\end{proof}

With the above lemma at hand, we can now apply the well-known approximation theorem for closed convex sets in uniformly convex spaces in order to get, for every given $f \in L^q(X,\mu,\R)$, the existence of a unique isotonic element of $L^q(X,\mu,I)$ that is closest to $f$ in $q$-norm. 

\begin{thm} \label{thm:ex-and-uniqueness-of-p}
Suppose $(X,\mathcal{A},\mu)$ is a complete measure space with a partial order $\le$, $I\subset \R$ is a non-empty closed interval, $q \in (1,\infty)$ and $f \in L^q(X,\mu,\R)$.
%Suppose $(X,\mathcal{A},\mu)$ is a complete measure space with a partial order $\le$ and $I\subset \R$ is a closed interval. Also, let $q \in (1,\infty)$ and $f \in L^q(X,\mu,\R)$. 
Then there exists a unique element $p(f) \in L^q_{\le}(X,\mu,I)$ such that 
\begin{align}
\norm{f-p(f)}_{q,\mu} = \inf_{g\in L^q_{\le}(X,\mu,I)} \norm{f-g}_{q,\mu}.
\end{align}
In other words, the functional $J_f|_{L^q_{\le}(X,\mu,I)}$ with
\begin{align}
J_f(g) := \norm{f-g}_{q,\mu}^q = \int_X |f-g|^q \d\mu \qquad (g \in L^q(X,\mu,\R))
\end{align}
has a unique minimizer $p(f)$. Additionally, every minimizing sequence for $J_f|_{L^q_{\le}(X,\mu,I)}$, that is, every sequence $(p_n)$ with 
\begin{align}
p_n \in   L^q_{\le}(X,\mu,I) 
\qquad \text{and} \qquad 
J_{f}(p_n) \longrightarrow \inf_{g\in L^q_{\le}(X,\mu,I)} J_{f}(g),
\end{align}
%with $p_n \in   L^q_{\le}(X,\mu,I)$ and with $J_{f}(p_n) \longrightarrow \inf_{g\in L^q_{\le}(X,\mu,I)} J_{f}(g)$, 
converges %in the norm~$\norm{\cdot}_{q,\mu}$ to the minimizer $p(f)$
to the minimizer $p(f)$ in the norm~$\norm{\cdot}_{q,\mu}$.
\end{thm}

\begin{proof}
Since $L^q(X,\mu,\R)$ with its standard norm $\norm{\cdot}_{q,\mu}$ is a uniformly convex Banach space for $q \in (1,\infty)$ 
by Clarkson's theorem (Theorem~5.2.11 of~\cite{Me}, for instance) and since $L^q_{\le}(X,\mu,I)$ is a non-empty closed convex subset of $L^q(X,\mu,\R)$ by Lemma~\ref{lm:L^q_sigma-convex}, the first part of the theorem follows by the well-known approximation theorem for closed convex sets in uniformly convex spaces (Corollary~8.2.1 of~\cite{La}, for instance). 
In order to prove the second part of the theorem, let $(p_n)$ be a minimizing sequence of $J_f|_{L^q_{\le}(X,\mu,I)}$ and write $h_n := p_n-f$ and $K := L^2_{\le}(X,\mu,I) - f$. Then $K$ is a non-empty closed convex subset of $L^q(X,\mu,\R)$ by Lemma~\ref{lm:L^q_sigma-convex} and
\begin{align} \label{eq:ex-and-uniqueness-of-p-1}
h_n \in K \qquad (n\in\N) 
\qquad \text{and} \qquad 
\norm{h_n}_{q,\mu} \longrightarrow \inf_{h\in K} \norm{h}_{q,\mu} \qquad (n\to\infty).
\end{align}
So, by the proof of Theorem~8.2.2 in~\cite{La}, the sequence $(h_n)$ converges w.r.t.~$\norm{\cdot}_{q,\mu}$ to the unique element $h_0 \in K$ with $\norm{h_0}_{q,\mu} = \inf_{h\in K}\norm{h}_{q,\mu}$ and, therefore, 
\begin{align} \label{eq:ex-and-uniqueness-of-p-2}
p_n = h_n + f \underset{\norm{\cdot}_{q,\mu}}{\longrightarrow} h_0 + f =: p \qquad (n\to\infty).
\end{align}
Since $p \in K + f = L^q_{\le}(X,\mu,I)$ and $\norm{f-p}_{q,\mu} = \inf_{g\in L^2_{\le}(X,\mu,I)} \norm{f-g}_{q,\mu}$, we see that $p = p(f)$ which, in conjunction with~\eqref{eq:ex-and-uniqueness-of-p-2}, proves the second part of the theorem.  %and the second part of the theorem follows. 
\end{proof}

As usual, we call the unique closest point $p(f)$ of $L^q_{\le}(X,\mu,I)$ to a given $f \in L^q(X,\mu,\R)$ the \emph{projection of $f$ onto $L^q_{\le}(X,\mu,I)$} or the \emph{isotonic ($q$-integrable) regression of $f$ (with values in $I$)}. %or the best approximation of $f$ out of $L^q_{\le}(X,\mu,I)$
%
%In the special case where $q=2$ and $I = \R$, this projection $p(f)$ has  nice geometric characterizations: %namely as the
We now restrict our attention to the case %$q=2$. %in the following.
\begin{align}
q=2 \qquad \text{and} \qquad I = \R
\end{align}
In this special case, the projection $p(f)$ can be nicely characterized in geometric terms: %has  nice geometric characterizations:
for instance, as the element $f^*$ of $L^2_{\le}(X,\mu,\R)$ for which the vector $f-f^*$ makes an obtuse angle with all the vectors $g-f^*$ with $g \in L^2_{\le}(X,\mu,\R)$.

\begin{prop} \label{prop:geometr-char-of-p}
Suppose $(X,\mathcal{A},\mu)$ is a complete measure space with a partial order $\le$ and let $f, f^* \in L^2(X,\mu,\R)$. %and let $q = 2$.
Then the following conditions are equivalent.
\begin{itemize}
\item[(i)] $f^*$ is the projection of $f$ onto $L^2_{\le}(X,\mu,\R)$, in short: $f^* = p(f)$
\item[(ii)] $f^* \in L^2_{\le}(X,\mu,\R)$ and $\scprd{f-f^*,f^*-g}_{2,\mu} \ge 0$ for all $g \in L^2_{\le}(X,\mu,\R)$
\item[(iii)] $f^* \in L^2_{\le}(X,\mu,\R)$ and $\scprd{f-f^*,f^*}_{2,\mu} = 0$ while $\scprd{f-f^*,g}_{2,\mu} \le 0$ for all $g \in L^2_{\le}(X,\mu,\R)$. 
%\begin{align}
%\scprd{f-f^*,f^*}_{2,\mu} = 0 
%\qquad \text{and} \qquad \scprd{f-f^*,g}_{2,\mu} \le 0 \qquad  (g \in L^2_{\le}(X,\mu,\R)).
%\end{align}
\end{itemize} 
\end{prop}

\begin{proof}
Since $L^2_{\le}(X,\mu,\R)$ is a closed convex set by Lemma~\ref{lm:L^q_sigma-convex} and also a cone, the stated equivalences immediately follow from~Theorem~8.2.2 and Theorem~8.2.7 of~\cite{RoWrDy}. %(or Theorem~1.4.2 and Corollary~1.4.2 in~\cite{Ba})
\end{proof}

In the next two propositions, we further specialize to finite measure spaces $(X,\mathcal{A},\mu)$ in order to ensure that $L^2(X,\mu,\R)$ contains the constant functions and the essentially bounded functions. %$L^{\infty}(X,\mu,\R)$

\begin{prop} \label{prop:elem-properties-of-p}
Suppose $(X,\mathcal{A},\mu)$ is a complete measure space with $\mu(X) < \infty$ and with a partial order $\le$. %and let $q = 2$. 
Then
\begin{itemize}
\item[(i)] %$p$ is invariant on $L^2_{\le}(X,\mu,\R)$, that is, 
$p(f) = f$ for every $f \in L^2_{\le}(X,\mu,\R)$ and, moreover, 
$p$ is idempotent and positively homogeneous, that is,
\begin{align*}
p\big(p(f)\big) = p(f) \qquad \text{and} \qquad p(\alpha f) = \alpha p(f)
\qquad (f \in L^2(X,\mu,\R), \alpha \in [0,\infty))
\end{align*}
\item[(ii)] $p$ is constant-preserving, that is $p(c) = c$ for every $c \in \R$ and, more generally, 
\begin{align*}
p(f+c) = p(f) + c \qquad (f \in L^2(X,\mu,\R), c \in \R)
\end{align*}
\item[(iii)] $p$ is integral-preserving, that is,
\begin{align*}
p\Big( \int_X f \d\mu \Big) = \int_X p(f) \d\mu \qquad (f \in L^2(X,\mu,\R)).
\end{align*}
\end{itemize} 
\end{prop}

\begin{proof}
Assertions~(i) and~(ii) are direct consequences of the definition of the projection operator. Assertion~(iii) easily follows from Proposition~\ref{prop:geometr-char-of-p}. Indeed, by the third condition of that proposition applied to the constant (hence isotonic) functions $g_{\pm} := \pm 1 \in L^2_{\le}(X,\mu,\R)$, %and by part~(ii) of the present proposition, 
we see that
\begin{align}
0 \ge \scprd{f-p(f),g_{\pm}}_{2,\mu} = \pm \bigg( \int_X f \d\mu - \int_X p(f)\d\mu \bigg) 
%=  \pm \bigg( p\Big( \int_X f \d\mu \Big) - \int_X p(f)\d\mu \bigg)
\end{align}
which in conjunction with the already proven constant-preservation property yields assertion~(iii).
\end{proof}

In the next proposition, we establish the contractivity of the projection operator $p$ both w.r.t.~the $2$-norm and w.r.t.~the $\infty$-norm. We will make essential use of these contractivity properties in our approximation results below.

\begin{prop} \label{prop:properties-of-p}
Suppose $(X,\mathcal{A},\mu)$ is a complete measure space with $\mu(X) < \infty$ and with a partial order $\le$. %and let $q = 2$. 
Then
\begin{itemize}
\item[(i)] $p$ is a monotonic operator, that is, $\scprd{p(f_1)-p(f_2),f_1-f_2}_{2,\mu} \ge 0$ for all $f_1,f_2 \in L^2(X,\mu,\R)$.
%\begin{align*}
%\scprd{p(f_1)-p(f_2),f_1-f_2}_{2,\mu} \ge 0 
%\qquad (f_1,f_2 \in L^2(X,\mu,\R)).
%\end{align*}
Additionally, $p$ is $\norm{\cdot}_{2,\mu}$-contractive, that is, 
\begin{align} \label{eq:properties-of-p-assertion-1}
\norm{p(f_1)-p(f_2)}_{2,\mu} \le \norm{f_1-f_2}_{2,\mu}
\qquad (f_1,f_2 \in L^2(X,\mu,\R)).
\end{align}
\item[(ii)] $p$ is an order-preserving operator, that is, $p(f_1) \le p(f_2)$ for $f_1,f_2 \in L^2(X,\mu,\R)$ with $f_1 \le f_2$. Additionally, %In particular, 
$p$ restricted to $L^{\infty}(X,\mu,\R)$ is $\norm{\cdot}_{\infty,\mu}$-contractive, that is, $p(L^{\infty}(X,\mu,\R)) \subset L^{\infty}(X,\mu,\R)$ and 
\begin{align} \label{eq:properties-of-p-assertion-2}
\norm{p(f_1)-p(f_2)}_{\infty,\mu} \le \norm{f_1-f_2}_{\infty,\mu}
\qquad (f_1,f_2 \in L^{\infty}(X,\mu,\R)).
\end{align}
\end{itemize}
\end{prop}

\begin{proof}
Assertion~(i) is an immediate consequence of Theorem~8.2.5 and Theorem~8.2.6 of~\cite{RoWrDy}. 
Also, the first part of assertion~(ii) (on order-preservation) is a consequence of Theorem~8.2.8 of~\cite{RoWrDy} because the canonical partial order $\le$ on $L^2(X,\mu,\R)$, induced by almost-everywhere inequality and used here, is easily seen to satisfy the compatibility requirements~(8.2.14) through (8.2.18) of the aforementioned theorem. %canonical lattice structure of $L^2(X,\mu,\R)$ is compatible with the vector space structure in the sense that ...
It remains to prove the second part of assertion~(ii) (on $\norm{\cdot}_{\infty,\mu}$-contractivity). So, let $f_1,f_2 \in L^{\infty}(X,\mu,\R)$, then %$-\norm{f_i}_{\infty,\mu} \le f_i \le \norm{f_i}_{\infty,\mu}$ %(inequalities in the $\mu$-almost-everywhere sense).  
\begin{align} \label{eq:properties-of-p-1}
-\norm{f_i}_{\infty,\mu} \le f_i \le \norm{f_i}_{\infty,\mu}
\quad \text{and} \quad
f_2 - \norm{f_1-f_2}_{\infty,\mu} \le f_1 \le f_2 + \norm{f_1-f_2}_{\infty,\mu}.
\end{align}
And therefore, by the order-preservation property just proven and the general constant-preservation property from Proposition~\ref{prop:elem-properties-of-p},  
\begin{gather}
-\norm{f_i}_{\infty,\mu} = p\big(-\norm{f_i}_{\infty,\mu}\big) \le p(f_i) \le p\big(\norm{f_i}_{\infty,\mu}\big) = \norm{f_i}_{\infty,\mu} \label{eq:properties-of-p-2} \\
p(f_2) - \norm{f_1-f_2}_{\infty,\mu} \le p(f_1) \le p(f_2) + \norm{f_1-f_2}_{\infty,\mu}. \label{eq:properties-of-p-3}
%p(f_2) - \norm{f_1-f_2}_{\infty,\mu} = p\big( f_2  - \norm{f_1-f_2}_{\infty,\mu} \big) \le p(f_1) \le p\big( f_2  + \norm{f_1-f_2}_{\infty,\mu} \big) \notag \\
%= p(f_2) + \norm{f_1-f_2}_{\infty,\mu}
\end{gather}
In view of~\eqref{eq:properties-of-p-2}, the inclusion $p(L^{\infty}(X,\mu,\R)) \subset L^{\infty}(X,\mu,\R)$ is now clear and \eqref{eq:properties-of-p-3}, in turn, immediately implies the estimate~\eqref{eq:properties-of-p-assertion-2}, as desired. 
\end{proof}

We close with some remarks on an alternative definition of the isotonic elements of $L^q(X,\mu,\R)$ in the spirit of~\cite{RoWrDy}. It is based on the collection %family %set $\ol{\mathcal{U}}_{\le}^{\mathrm{mb}}(X)$ 
\begin{align} \label{eq:def-ol-U^mb}
\ol{\mathcal{U}}_{\le}^{\mathrm{mb}}(X) := \big\{ E \in \mathcal{A}: \mu\big((E\setminus U) \cup (U\setminus E)\big) = 0 \text{ for some } U \in \mathcal{U}_{\le}^{\mathrm{mb}}(X) \big\}
\end{align}
of all measurable sets that differ from a measurable upper set $U \in \mathcal{U}_{\le}^{\mathrm{mb}}(X) := \mathcal{U}_{\le}(X) \cap \mathcal{A}$ only by a $\mu$-null set. 
In analogy to~\cite{RoWrDy} (Definition~8.1.2) (where the case $q=2$ is treated), one could then define
\begin{align} \label{eq:altern-def-isotonic-el-of-L^q}
L^q\big( \ol{\mathcal{U}}_{\le}^{\mathrm{mb}}(X),\R \big) := \big\{ f \in L^q(X,\mu,\R): &\text{ for some (hence every) representative } f_0 \text{ of } f, \notag \\ 
&\text{ } f_0^{-1}((c,\infty)) \in \ol{\mathcal{U}}_{\le}^{\mathrm{mb}}(X) \text{ for all } c \in \R \big\}
\end{align} 
and show that this alternative definition of isotonicity for elements of $L^q(X,\mu,\R)$ is actually equivalent to our definition~\eqref{eq:def-isotonic-el-of-L^q}, in short:
\begin{align} \label{eq:equivalence-of-defs-of-isotonic-el-of-L^q}
L^q\big( \ol{\mathcal{U}}_{\le}^{\mathrm{mb}}(X),\R \big) = L^q_{\le}(X,\mu,\R) \qquad (q\in [1,\infty)).  
\end{align}
We opted for %decided for 
the definition~\eqref{eq:def-isotonic-el-of-L^q} because it is simpler and simpler to work with later on. 
In order to see the forward inclusion in~\eqref{eq:equivalence-of-defs-of-isotonic-el-of-L^q}, note that for every $f \in L^q\big( \ol{\mathcal{U}}_{\le}^{\mathrm{mb}}(X),\R \big)$ and every representative $f_0$ of $f$, all characteristic functions of the form $\chi_{f_0^{-1}((c,\infty))}$ are  essentially isotonic by Lemma~\ref{lm:char-isotonic-fcts} and that $f$ can be approximated in $\norm{\cdot}_{q,\mu}$ by simple functions of the form
\begin{gather*}
\phi_0 = \sum_{k=-m}^0 c_k \chi_{f_0^{-1}((c_{k-1},c_k])} + \sum_{k=0}^{m} c_k \chi_{f_0^{-1}((c_k,c_{k+1}])} 
%\\
%=
%\sum_{k=-m}^{-1} (c_k-c_{k+1}) \chi_{f_0^{-1}((-\infty,c_k])} + \sum_{k=1}^m (c_k-c_{k-1}) \chi_{f_0^{-1}((c_k,\infty))} 
\end{gather*}
with $k \mapsto c_k$ monotonically increasing and $c_{-m-1} := -\infty$, $c_0 = 0$, $c_{m+1} := \infty$. Since
\begin{align*}
\phi_0 = \sum_{k=-m}^{-1} (c_k-c_{k+1}) \chi_{f_0^{-1}((-\infty,c_k])} + \sum_{k=1}^m (c_k-c_{k-1}) \chi_{f_0^{-1}((c_k,\infty))},
\end{align*}
the simple functions of that form are essentially isotonic and hence belong to $L^q_{\le}(X,\mu,\R)$, as desired. 
In order to see the backward inclusion in~\eqref{eq:equivalence-of-defs-of-isotonic-el-of-L^q}, note that for every $f \in L^q_{\le}(X,\mu,\R)$ and every representative $f_0$ of $f$, there is a null set $N$ such that 
\begin{align*}
f_0^{-1}((c,\infty)) \cap X\setminus N = \big(f_0|_{X\setminus N}\big)^{-1}((c,\infty)) \in \mathcal{U}_{\le}(X\setminus N)
\end{align*}
by Lemma~\ref{lm:char-isotonic-fcts} and that any $U_0 \in \mathcal{U}_{\le}(X\setminus N)$ can be extended to a $U \in \mathcal{U}_{\le}(X)$ by setting
\begin{align} \label{eq:extend-upper-set-on-X-minus-N-to-upper-set-on-X}
U := U_0 \cup \{y\in N: y \ge u \text{ for some } u \in U \}.
\end{align}
(Incidentally, in the special case of finite measures $\mu$ with $\mu(X) = 1$ and $q=2$, the identity~\eqref{eq:equivalence-of-defs-of-isotonic-el-of-L^q} could also be inferred from Theorem~8.3.2 of~\cite{RoWrDy} due to %in view of 
the properties of our projection operator $p$ established in Proposition~\ref{prop:elem-properties-of-p} and~\ref{prop:properties-of-p}.)

\subsection{Case of rectangular multivariate domains}%{Case of rectangular multivariate sets}

We now further specialize to the case where $X$ is a -- continuous or discrete -- rectangular set in $\R^d$ and where $\mu$ is a weighted Lebesgue or counting measure, respectively. Specifically, %this means that
\begin{align}
X = Q = Q_1 \times \dotsb \times Q_d 
\qquad \text{or} \qquad
X = G = G_1 \times \dotsb \times G_d
\end{align} 
with compact intervals $Q_i = [a_i,b_i] \subset \R$ with $a_i < b_i$ or finite sets $G_i \subset \R$ with $G_i \ne \emptyset$, 
\begin{align}
\mathcal{A} = \mathcal{Q}_Q \qquad \text{or} \qquad \mathcal{A} = \mathcal{P}_G
\end{align}
(Lebesgue sigma-algebra on $Q$ and power set of $G$, respectively), and 
\begin{gather}
\mu(E) = (w_Q \cdot \lambda)(E) = 
\int_E w_Q \d\lambda = \int_E w_Q(x)\d x \qquad (E \in \mathcal{Q}_Q) \\
\text{or} \notag \\
\mu(E) = (w_G \cdot \kappa)(E) = 
\int_E w_G \d\kappa = \sum_{x\in E} w_G(x) \qquad (E \in \mathcal{P}_G),
\end{gather}
where $\lambda$, $\kappa$ are Lebesgue or counting measure on $\mathcal{Q}_Q$ or $\mathcal{P}_G$, respectively, and $w_Q, w_G$ are $\mathcal{A}$-measurable weight functions on $X$ that are bounded above and below
\begin{align} \label{eq:weight-fcts-bd-above-and-below}
\ul{c} \le w_Q(x) \le \ol{c} \qquad (x\in Q) 
\qquad \text{and} \qquad 
\ul{c} \le w_G(x) \le \ol{c} \qquad (x\in G) 
\end{align}
by positive constants $\ul{c},\ol{c} \in (0,\infty)$. 
Also, the partial order %we consider
on $X$ will be given by $\le_{\sigma}$ as defined in~\eqref{eq:def-le_sigma} above -- with some $\sigma \in \{-1,0,1\}^d$. 
In view of~\eqref{eq:weight-fcts-bd-above-and-below}, the set $L^2(X,\mu,\R)$ is independent of the weight function $w_Q, w_G$:
\begin{gather*}
L^2(Q,w_Q \cdot \lambda,\R) = L^2(Q,\lambda,\R) =: L^2(Q,\R), \\
L^2(G,w_G \cdot \kappa,\R) = L^2(G,\kappa,\R) =: \ell^2(G,\R).
\end{gather*}
%\begin{align*}
%L^2(Q,w_Q \lambda,\R) = L^2(Q,\lambda,\R) =: L^2(Q,\R),
%\qquad 
%L^2(G,w_G \kappa,\R) = L^2(G,\kappa,\R) =: \ell^2(G,\R).
%\end{align*}
Accordingly, we also have $L^2_{\sigma}(Q, w_Q\cdot \lambda,I) = L^2_{\sigma}(Q,\lambda,I) =: L^2_{\sigma}(Q,I)$ and $L^2_{\sigma}(Q,w_G \cdot \kappa,I) = L^2_{\sigma}(G,\kappa,I) =: \ell^2_{\sigma}(G,I)$, where 
\begin{align}
L^2_{\sigma}(X,\mu,I) := L^2_{\le_{\sigma}}(X,\mu,I). 
\end{align}
While $L^2(X,\mu,\R)$ as a set is independent of the weight $w_Q$ or $w_G$, the canonical %weighted 
scalar product $\scprd{\cdot,\cdot\cdot}_{2,w_Q}$ or $\scprd{\cdot,\cdot\cdot}_{2,w_G}$ of $L^2(X,\mu,\R)$ does depend on the weight, of course. Correspondingly, the projection of a given $f_Q \in L^2(Q,\R)$ or $f_G \in \ell^2(G,\R)$ onto $L^2_{\sigma}(Q,\R)$ or $\ell^2_{\sigma}(G,\R)$, that is, the  minimizer of %the functional 
$J_{f_Q,w_Q}|_{L^2_{\sigma}(Q,\R)}$ or $J_{f_G,w_G}|_{\ell^2_{\sigma}(G,\R)}$ with 
\begin{align}
J_{f_Q,w_Q}(g) := \norm{f_Q-g}_{2,w_Q}^2 
\qquad \text{or} \qquad
J_{f_G,w_G}(g) := \norm{f_G-g}_{2,w_G}^2, 
\end{align}
depends on the weight function as well. In the following, we will denote this projection by $p_{\sigma}^{w_Q}(f_Q)$ or $p_{\sigma}^{w_G}(f_G)$ and refer to it as the \emph{$\sigma$-monotonic regression} of $f_Q$ or $f_G$, respectively.
We will also use the notation
\begin{align*}
\Av_{f_Q,w_Q}(E) := \frac{\int_E f_Q(x) w_Q(x) \d x}{\int_E w_Q(x) \d x}
%\frac{\int_E f_Q w_Q \d\lambda}{\int_E w_Q \d\lambda}
\qquad \text{and} \qquad
\Av_{f_G,w_G}(E) := \frac{\sum_{x\in E} f_G(x) w_G(x)}{\sum_{x\in E} w_G(x)} 
\end{align*}
for the $(w_Q \cdot \lambda)$- or $(w_G \cdot \kappa)$-average %w.r.t.~the measure $w_Q \cdot \lambda$ or $w_G \cdot \kappa$ 
of the function $f_Q \in L^2(Q,\R)$ or $f_G \in \ell^2(G,\R)$ over any non-null set
\begin{align*}
E \in \mathcal{Q}_Q^{>} := \{E\in\mathcal{Q}_Q: \lambda(E) >0\}
\qquad \text{or} \qquad 
E \in \mathcal{P}_G^{>} := \{E\in\mathcal{P}_G: E \ne \emptyset \},
\end{align*}
%of positive measure. 
respectively. 
At a few places, we will need the following lemma concerning the transition from $\sigma$ to $-\sigma$ and, in that context (and throughout the paper), we use the short-hand notations
\begin{align}
\mathcal{L}_{\sigma} := \mathcal{L}_{\le_{\sigma}}(Q)
\qquad \text{and} \qquad 
\mathcal{U}_{\sigma} := \mathcal{U}_{\le_{\sigma}}(Q).
\end{align}
We omit the straightforward proof of the lemma. %Its proof is a straightforward consequence of the definitions and we thus omit it. 

\begin{lm} \label{lm:minus-sigma}
Suppose $f \in L^2(Q,\R)$ and $w \in L^{\infty}(Q,[c,\infty))$ with some $c \in (0,\infty)$. Then
\begin{multicols}{2} %\raggedcolumns
\begin{itemize}
\item[(i)] $p_{\sigma}^w(f) = -p_{-\sigma}^w(-f)$
\item[(ii)] $\mathcal{U}_{\sigma} = \mathcal{L}_{-\sigma}$. 
%\item[(ii)] $\mathcal{U}_{\sigma} = \mathcal{L}_{-\sigma} = \{Q\setminus L: L \in \mathcal{L}_{\sigma}\}$ and $\chi_U$ is $\sigma$-monotonic for every $U \in \mathcal{U}_{\sigma}$.
\end{itemize}
\end{multicols}
\end{lm}

In our continuity result below, we will finally need another %a last
straightforward lemma, which relates monotonicity w.r.t.~a general signature $\sigma \in \{-1,0,1\}^d$ to monotonicity w.r.t.~a signature having only $\pm 1$ as components. %in $\{-1,1\}^{\hat{d}}$. %It is convenient, in that context, to adopt the following notations 
In that context, it is convenient to adopt the following notations for a rectangular set $Q := [a_1,b_1] \times \dotsb \times [a_d,b_d]$, a signature $\sigma \in \{-1,0,1\}^d$, and an arbitrary element $x \in Q$:
\begin{gather}
\hat{Q} := \bigtimes_{i\in\hat{I}} [a_i,b_i], \qquad \dot{Q} := \bigtimes_{i\in\dot{I}} [a_i,b_i], \qquad
\hat{\sigma} := (\sigma_i)_{i\in\hat{I}}, \\
%, \qquad \dot{\sigma} := (\sigma_i)_{i\in\dot{I}} = 0
\hat{x} = (x_i)_{i\in\hat{I}}, \qquad \dot{x} := (x_i)_{i\in\dot{I}},
\qquad \hat{x} \& \dot{x} := x,
\end{gather} 
where $\hat{I} := \{i \in \{1,\dots,d\}: \sigma_i = \pm 1\}$ and $\dot{I} := \{ i \in \{1,\dots,d\}: \sigma_i = 0\}$.

\begin{lm} \label{lm:sigma-hat}
A function $f:Q \to \R$ is $\sigma$-monotonic if and only if the functions $f(\cdot\,\&\dot{x}): \hat{Q} \to \R$ are $\hat{\sigma}$-monotonic for all $\dot{x} \in \dot{Q}$. 
\end{lm}

\section{Special case of grid-constant functions}

In this section, we consider the special case of grid-constant functions $f$ and $w$ on $Q$, that is, functions which are constant on the cells of a grid $G$ on $Q$, and relate the $\sigma$-monotonic regression $p_{\sigma}^w(f)$ to its discrete counterpart, namely the $\sigma$-monotonic regression $p_{\sigma}^{w|_G}(f|_G)$ of the functions $f,w$  restricted to the grid $G$.  
\smallskip

We begin by properly defining what we mean by grids, grid cells, and grid-constant functions. As usual, let $Q$ be a compact rectangular domain of the form
\begin{align}
Q = [a_1,b_1] \times \dotsb \times [a_d,b_d] \subset \R^d
\end{align}
with $a_i < b_i$. A set $G$ is then called a \emph{grid on $Q$} iff there are partitions $P_i = \{t_{ik}: k\in \{0,\dots,m_i\}\}$ of the intervals $[a_i,b_i]$ such that
\begin{align}
G = G_1 \times \dots \times G_d 
\end{align}
with $G_i =\{(t_{ik-1}+t_{ik})/2: k \in \{1,\dots,m_i\}\}$. In other words, a grid on $Q$ is nothing but the set of midpoints of the cells of a rectangular partition $P = P_1 \times \dotsb \times P_d$ of $Q$. (As usual, a partition of a $1$-dimensional interval $[a_i,b_i]$ is a finite set of points $t_{ik} \in [a_i,b_i]$ with $a_i = t_{i0} < t_{i1} < \dotsb < t_{im_{i-1}} < t_{im_i} = b_i$.) A grid $G$ will be called \emph{equidistant} iff its defining partitions $P_i$ are all equidistant, that is,
\begin{align}
t_{ik}-t_{ik-1} = (b_i-a_i)/m_i \qquad (k \in \{1,\dots, m_i\} \text{ and } i \in \{1,\dots,d\}). 
\end{align}  
In particular, %Similarly, 
a grid $G$ will be called \emph{dyadic} iff %there is an $n \in \N$ such that 
its defining partitions $P_i$ consist of $m_i + 1 = 2^n + 1$ partition points and
\begin{align}
t_{ik}-t_{ik-1} = (b_i-a_i)/2^n \qquad (k \in \{1,\dots, 2^n+1\} \text{ and } i \in \{1,\dots,d\}) 
\end{align}
with some $i$-independent number $n \in \N$. 
Additionally, by a \emph{cell of a grid $G$ on $Q$} (with defining partitions $P_i$) we mean a rectangular  set $C$ of the form
\begin{align}
C = I_1 \times \dots \times I_d
\end{align}
where $I_i$ for every $i \in \{1,\dots,d\}$ is a lower semiclosed and upper relatively semiopen partition interval of the partition $P_i$, that is,
\begin{align}
I_i = [t_{ik-1},t_{ik}) \text{ for some } k \in \{1, \dots, m_i-1\} 
\qquad \text{or} \qquad
I_i = [t_{im_i-1},t_{i m_i}].
\end{align}
Clearly, for any grid $G$ on $Q$, the set $Q$ is the disjoint union of all cells of $G$, and for any given $x \in Q$ we will denote the $G$-cell containing $x$ by $C^G(x)$. Also, if $G$ is an equidistant grid on $Q$, then the interiors of its cells are translates of each other
\begin{align} \label{eq:cells-translates-of-each-other}
\operatorname{int} C^G(x) = \operatorname{int} C^G(y) + (x-y) \qquad (x,y \in G)
\end{align}
and we will write $\operatorname{len}_G$ for the maximal edge length and $\operatorname{vol}_G$ for the volume of the $G$-cells.

\begin{lm} \label{lm:cells}
Suppose $G$ is a grid on $Q$. 
\begin{itemize}
\item[(i)] If $x,y \in Q$ satisfy $x \le_{\sigma} y$, then also $x_G \le_{\sigma} y_G$, where $x_G, y_G \in G$ are the grid points with $C^G(x_G) \ni x$ and $C^G(y_G) \ni y$. 
\item[(ii)] If $S$ is a union of cells of $G$, then $S = \bigcup_{x \in S\cap G} C^G(x)$.
%\begin{align}
%S = \bigcup_{x \in S\cap G} C^G(x).
%\end{align}
\end{itemize}
\end{lm}

We omit the straightforward proof of this lemma. 
If $G$ is a grid on $Q$, then a function $f: Q \to \R$ will be called \emph{$G$-constant} iff it is constant on the cells of $G$, that is, the restriction $f|_{C}$ is constant for every cell $C$ of $G$ so that, in particular,
\begin{align*}
f = \sum_{x\in G} f(x) \chi_{C^G(x)}.
\end{align*}
Analogously, an equivalence class of functions on $Q$ is called \emph{$G$-constant} iff it has a $G$-constant representative. And finally, a function or equivalence class of functions on $Q$ is called \emph{(equidistantly or dyadically) grid-constant} iff it is $G$-constant for some (equidistant or dyadic) grid $G$ on $Q$.
With these preparations at hand, we can now state and prove the main result of this section. %which reduces the computation of monotonic regression functions for equidistantly grid-constant functions $f$ and $w$ to the computation of corresponding discrete monotonic regressions.  

\begin{thm} \label{thm:spec-case-grid-const-fct}
Suppose $G$ is an equidistant grid on $Q$ and $f:Q\to\R$ and $w:Q\to (0,\infty)$ are $G$-constant functions. Then $p_{\sigma}^w(f)$ is $G$-constant with
\begin{align} \label{eq:spec-case-grid-const-fct}
p_{\sigma}^w(f) = \sum_{x\in G} p_{\sigma}^{w|_G}(f|_G)(x) \chi_{C^G(x)}.
\end{align} 
In more precise terms, the function on the right-hand side of~\eqref{eq:spec-case-grid-const-fct} is a $\sigma$-monotonic representative of $p_{\sigma}^w(f)$. 
\end{thm}

\begin{proof}
Write $f_| := f|_G$, $w_| := w|_G$ and $f_|^* := p_{\sigma}^{w_|}(f_|)$ for brevity. Also, write 
\begin{align} \label{eq:spec-case-grid-const-fct-1}
f_0^* := \sum_{x\in G} f_|^*(x) \chi_{C^G(x)}
\end{align}
for the function on the right-hand side of~\eqref{eq:spec-case-grid-const-fct} as well as $f^*$ for the equivalence class of $f_0^*$. We will show that
\begin{align} \label{eq:spec-case-grid-const-fct-2}
f^* \in L^2_{\sigma}(Q,\R) \qquad \text{and} \qquad \scprd{f-f^*,f^*-g}_{2,w} \ge 0 \qquad (g \in L^2_{\sigma}(Q,\R)), 
\end{align}
from which the desired conclusion immediately follows by virtue of Propositon~\ref{prop:geometr-char-of-p}. 
Clearly, $f_0^*$ is a $G$-constant function and, by the $\sigma$-monotonicity of $f_|^*$ and Lemma~\ref{lm:cells}(i), $f_0^*$ is also $\sigma$-monotonic. In particular, (\ref{eq:spec-case-grid-const-fct-2}.a) is satisfied and it remains to establish~(\ref{eq:spec-case-grid-const-fct-2}.b). 
So, let $g \in L^2_{\sigma}(Q,\R)$ and let $g_0$ be any representative of $g$. Since $f$, $f_0^*$ and $w$ are $G$-constant and $G$ is equidistant,
we have
\begin{gather} 
\scprd{f-f^*,f^*-g}_{2,w} 
= \sum_{x\in G} \int_{C^G(x)} (f-f_0^*)(f_0^*-g_0) w \d\lambda \notag \\
= \sum_{x\in G} (f(x)-f_0^*(x))f_0^*(x) w(x) \operatorname{vol}_G
- \sum_{x\in G} (f(x)-f_0^*(x)) \Big( \int_{C^G(x)} g_0 \d\lambda \Big) w(x) \notag \\
= \big\langle f_|-f_|^*,f_|^* \big\rangle_{2,w_|} \cdot \operatorname{vol}_G \, - \,\, \big\langle f_|-f_|^*,h_| \big\rangle_{2,w_|},
\label{eq:spec-case-grid-const-fct-3}
\end{gather}
where $h_|(x) := \int_{C^G(x)} g_0 \d\lambda$ for $x \in G$. Since $g_0$ is essentially $\sigma$-monotonic, $h_|$ is $\sigma$-monotonic as well. (Indeed, by the equidistance of $G$, we have~\eqref{eq:cells-translates-of-each-other}
%the interiors of the $G$-cells are translates of each other: %the cells of $G$ are essentially translates of each other in the sense that
%\begin{align}
%C^G(x)^{\circ} = C^G(y)^{\circ} + (x-y) \qquad (x,y \in G)
%\end{align}   
and therefore for $x,y \in G$ with $x\le_{\sigma} y$ we have
\begin{align*}
h_|(x) = \int_{C^G(y)+(x-y)} g_0(z) \d z = \int_{C^G(y)} g_0(z+x-y) \d z
\le \int_{C^G(y)} g_0(z) \d z = h_|(y),
\end{align*}
as desired.) So, by virtue of Proposition~\ref{prop:geometr-char-of-p}, %with $(X,\mathcal{A},\mu) := (G,\mathcal{P}_G, w_| \cdot \kappa)$, 
we obtain %it follows
\begin{align} \label{eq:spec-case-grid-const-fct-4}
\big\langle f_|-f_|^*,f_|^* \big\rangle_{2,w_|} = 0
\qquad \text{and} \qquad
\big\langle f_|-f_|^*,h_| \big\rangle_{2,w_|} \le 0
\end{align}
which in conjunction with~\eqref{eq:spec-case-grid-const-fct-3} yields~(\ref{eq:spec-case-grid-const-fct-2}.b), as desired. 
\end{proof}

Combining the above result with the well-known discrete averaging formulas (Theorem~1.4.4 of~\cite{RoWrDy}), we obtain averaging formulas for the monotonic regression $p_{\sigma}^w(f)$ of grid-constant functions $f$ and $w$. We will use the abbreviations
\begin{align}
a_{\sigma,\mathrm{is}}^{w,G}(f)(x) &:= \inf_{L \in \mathcal{L}_{\sigma}^G(x)} \sup_{U \in \mathcal{U}_{\sigma}^G(x)} \Av_{f,w}(L \cap U) \\
a_{\sigma,\mathrm{si}}^{w,G}(f)(x) &:= \sup_{U \in \mathcal{U}_{\sigma}^G(x)} \inf_{L \in \mathcal{L}_{\sigma}^G(x)} \Av_{f,w}(L \cap U),
\end{align}
where $\mathcal{L}_{\sigma}^G(x) := \{L \in \mathcal{L}_{\sigma}: L \text{ is a union of $G$-cells and } L \ni x\}$ and $\mathcal{U}_{\sigma}^G(x) := \{U \in \mathcal{U}_{\sigma}: U \text{ is a union of $G$-cells and } U \ni x\}$

\begin{cor} \label{cor:avg-G}
Suppose $G$ is an equidistant grid on $Q$ and $f:Q\to\R$ and $w:Q\to (0,\infty)$ are $G$-constant functions. Then $a_{\sigma,\mathrm{is}}^{w,G}(f)$ and $a_{\sigma,\mathrm{si}}^{w,G}(f)$ are $\sigma$-monotonic representatives of $p_{\sigma}^w(f)$.  
\end{cor}

\begin{proof}
Write $f_| := f|_G$, $w_| := w|_G$ and $f_|^* := p_{\sigma}^{w_|}(f_|)$ for brevity. Also, write 
\begin{align} \label{eq:avg-G-1}
f_0^* := \sum_{x\in G} f_|^*(x) \chi_{C^G(x)}.
\end{align}
We know that $f_0^*$ is a $\sigma$-monotonic representative of $p_{\sigma}^w(f)$ by the previous theorem and that $f_|^*$ can be expressed as %in terms of discrete averages
\begin{align} \label{eq:avg-G-2}
&\inf_{L^G \in \mathcal{L}_{\sigma}(G,x_0)} \sup_{U^G \in \mathcal{U}_{\sigma}(G,x_0)} \Av_{f_|,w_|}(L^G \cap U^G) 
= f_|^*(x_0)  \notag \\
&\qquad \qquad = \sup_{U^G \in \mathcal{U}_{\sigma}(G,x_0)} \inf_{L^G \in \mathcal{L}_{\sigma}(G,x_0)} \Av_{f_|,w_|}(L^G \cap U^G) 
\qquad (x_0 \in G)
\end{align}
by the well-known discrete averaging theorem (Theorem~1.4.4 of~\cite{RoWrDy}), where $\mathcal{L}_{\sigma}(G,x_0)$ and $\mathcal{U}_{\sigma}(G,x_0)$ is, respectively, the set of all $\sigma$-lower or $\sigma$-upper subsets of $G$ that contain $x_0$.
With the help of Lemma~\ref{lm:cells}, it easily follows that
\begin{align} \label{eq:avg-G-3}
\mathcal{L}_{\sigma}(G,x_0) = \big\{ L\cap G: L \in \mathcal{L}_{\sigma}^G(x_0) \big\}, 
\qquad %\qquad \text{and} \qquad
\mathcal{U}_{\sigma}(G,x_0) = \big\{ U\cap G: U \in \mathcal{U}_{\sigma}^G(x_0) \big\}
\end{align} 
for every $x_0 \in G$. 
Also, for every $L \in \mathcal{L}_{\sigma}^G(x_0)$ and $U \in \mathcal{U}_{\sigma}^G(x_0)$ with $x_0 \in G$, the intersection $L\cap U$ is a union of $G$-cells and thus
\begin{align} \label{eq:avg-G-4}
L \cap U = \bigcup_{x\in L\cap U \cap G} C^G(x) \supset C^G(x_0) 
\qquad (L \in \mathcal{L}_{\sigma}^G(x_0), U \in \mathcal{U}_{\sigma}^G(x_0), x_0 \in G)
\end{align}
(Lemma~\ref{lm:cells}). 
Since $f$ and $w$ are $G$-constant  and $G$ is equidistant, we see by~(\ref{eq:avg-G-4}.a) that
\begin{align*} 
\int_{L\cap U} fw \d\lambda = \sum_{x\in L\cap U\cap G} f(x)w(x) \cdot \operatorname{vol}_G
\qquad \text{and} \qquad
\int_{L\cap U} w \d\lambda = \sum_{x\in L\cap U\cap G} w(x) \cdot \operatorname{vol}_G
\end{align*}
and therefore
\begin{align} \label{eq:avg-G-5}
\Av_{f,w}(L\cap U) = \Av_{f_|,w_|}(L\cap U\cap G)
\qquad (L \in \mathcal{L}_{\sigma}^G(x_0), U \in \mathcal{U}_{\sigma}^G(x_0), x_0 \in G),
\end{align}
where we used that $L\cap U$ is non-null by (\ref{eq:avg-G-4}.b). Combining now~\eqref{eq:avg-G-1}, \eqref{eq:avg-G-2}, \eqref{eq:avg-G-3} and \eqref{eq:avg-G-5}, we obtain
\begin{align} \label{eq:avg-G-6}
a_{\sigma,\mathrm{is}}^{w,G}(f)(x_0) = f_0^*(x_0) = a_{\sigma,\mathrm{si}}^{w,G}(f)(x_0) \qquad (x_0 \in G).
\end{align}
Since the maps $Q \ni x \mapsto \mathcal{L}_{\sigma}^G(x), \mathcal{U}_{\sigma}^G(x)$ are easily seen to be constant on the cells of $G$ and $f_0^*$ is $G$-constant as well, the equalities in~\eqref{eq:avg-G-6} extend to arbitrary $x \in Q$. Consequently, $a_{\sigma,\#}^{w,G}(f) = f_0^*$ is a $\sigma$-monotonic representative of $p_{\sigma}^w(f)$ for $\# \in \{\mathrm{is}, \mathrm{si}\}$, as desired.    
\end{proof}

\section{Approximation of monotonic regression functions by grid-constant monotonic regression functions}%{Approximation of monotonic regressors by grid-constant monotonic regressors}

In this section, we establish two results concerning the approximation of general monotonic regression functions $p_{\sigma}^w(f)$ by monotonic regression functions $p_{\sigma}^{w_n}(f_n)$ of dyadically grid-constant functions $f_n$ and $w_n$. In the first result, $f$ is a general square-integrable function and the approximation is w.r.t.~the $2$-norm, while in the second result $f$ is even essentially bounded and the approximation is w.r.t.~the $\infty$-norm. In both results, we establish simple upper bounds on the approximation error in the respective norm. 
\smallskip

In view of the results from the previous section, the approximants $p_{\sigma}^{w_n}(f_n)$ are dyadically grid-constant themselves and they are determined by the corresponding discrete monotonic regression functions. %$p_{\sigma}^{w_n|_{G_n}}(f_n|_{G_n})$. 
So, with the help of the approximation results below and the results from the previous section, we can reduce the computation of general monotonic regression functions $p_{\sigma}^w(f)$ to the computation of discrete monotonic regression functions and, for that purpose, in turn we can use any of the known discrete algorithms from the literature~\cite{BaBaBrBr} (Section~2.3), \cite{RoWrDy} (Section~1.4), \cite{DyRo82}, \cite{Dy83}, \cite{MaMu85}, \cite{QiEd96}, \cite{HoQu03}, \cite{SpWaWi03}, \cite{St13}, \cite{St15}, \cite{KyRaSa15} and~\cite{St19}.
\smallskip

In the paper~\cite{LiDu14}, a completely different algorithm is proposed to compute $p_{\sigma}^w(f)$, for continuous $f$ and $w \equiv 1$ and $\sigma = (1, \dots, 1)$. It is inspired by the algorithm from~\cite{DyRo82}, \cite{Dy83} and, in a nutshell, works as follows: it again and again cycles %runs
through all coordinate directions and in each direction applies a univariate monotonic regression in that direction (using the pool-adjacent-violators algorithm). In the limit of infinitely many cycles, this algorithm converges to $p_{\sigma}^w(f)$ (Theorem~1 of~\cite{LiDu14}), but no information on the rate of  convergence is given in~\cite{LiDu14}. Additionally, the algorithm from~\cite{LiDu14} is not as simple as our general computational methodology. %neither conceptionally nor computationally  
We will use the abbreviations
\begin{align}
S(Q,I) := \{f: f \text{ is a dyadically grid-constant function on } Q \text{ with values in } I \}
\end{align}
and $S_{\sigma}(Q,I) := \{f \in S(Q,I): f \text{ is $\sigma$-monotonic} \}$ along with the abbreviations 
\begin{align}
S^{\infty}(Q,\R) := \ol{S(Q,\R)}^{\norm{\cdot}_{\infty}}
\qquad \text{and} \qquad
S^{\infty}(Q,(0,\infty)) := \bigcup_{c \in (0,\infty)} \ol{S(Q,[c,\infty))}^{\norm{\cdot}_{\infty}}
\end{align}
and $S_{\sigma}^{\infty}(Q,\R) := \{f \in S^{\infty}(Q,\R): f \text{ is $\sigma$-monotonic} \}$.

\begin{lm} \label{lm:S^infty}
$S^{\infty}(Q,(0,\infty)) \supset C(Q,(0,\infty))$ and for every $w \in S^{\infty}(Q,(0,\infty))$ there exist positive constants $\ul{c},\ol{c} \in (0,\infty)$ such that $\ul{c} \le w(x) \le \ol{c}$ for almost every $x \in Q$. 
\end{lm}

\begin{proof}
Clearly, the second part of the lemma is an immediate consequence of the definition of $S^{\infty}(Q,(0,\infty))$ and we therefore prove only the first part. So, let $f \in C(Q,(0,\infty))$ and define
\begin{align}
f_n := \sum_{x\in G_n} f(x) \chi_{C^{G_n}(x)}
\end{align}
where $(G_n)$ is a sequence of dyadic grids on $Q$ with $\operatorname{len}_{G_n} \longrightarrow 0$. Since $f$ is uniformly continuous on the compact set $Q$, we see that
\begin{align}
c := \inf_{x\in Q} f(x) > 0 
\qquad \text{and} \qquad
S(Q,[c,\infty)) \ni f_n \underset{\norm{\cdot}_{\infty}}{\longrightarrow} f \qquad (n\to\infty)
\end{align}
and therefore $f \in S^{\infty}(Q,(0,\infty))$, as desired. 
\end{proof}

\subsection{Approximation in the $2$-norm}

In this subsection, we prove our approximation result w.r.t.~the $2$-norm. We need two lemmas for that purpose. 

\begin{lm} \label{lm:S-dense-in-L^2}
$S(Q,\R)$ is dense in $L^2(Q,\R)$. 
%If $f \in L^2(Q,\R)$, then there exist dyadic grids $G_n$ on $Q$ and $G_n$-constant functions $f_n$ such that %$f_n \underset{\norm{\cdot}_2}{\longrightarrow} f$ as $n \to \infty$. 
%\begin{align}
%f_n \underset{\norm{\cdot}_2}{\longrightarrow} f \qquad (n\to\infty).
%\end{align}
\end{lm}

\begin{proof}
It is well-known that the set of dyadically constant functions on $\R^d$ -- that is, the functions on $\R^d$ which for some $m \in \N$ are constant on all cells of the form $[(k_1-1)/2^m,k_1/2^m) \times \dotsb \times [(k_d-1)/2^m,k_d/2^m)$ with $k \in \Z^d$ -- is dense in $L^2(\R^d,\R)$. (See, for instance, Lemma~2.17 of~\cite{LiLo} and its proof or Lemma~3.4.6 in conjunction with Proposition~1.4.1 and Lemma~1.4.2 of~\cite{Co}.) Consequently, we see by restriction to $[0,1]^d$ that $S([0,1]^d,\R)$ is dense in $L^2([0,1]^d,\R)$. Applying then the bijective affine transformation 
\begin{align*}
[0,1]^d \ni x \mapsto (a_1+x_1(b_1-a_1), \dots, a_d+x_d(b_d-a_d)) \in Q,
\end{align*}
we obtain the desired density assertion.
\end{proof}

\begin{lm} \label{lm:p^w(f)-contin-in-w,f}
Suppose $f_n,f \in L^2(Q,\R)$ and $w_n, w \in L^{\infty}(Q,[\ul{c},\infty))$ for some $\ul{c} \in (0,\infty)$ such that $f_n \longrightarrow f$ w.r.t.~$\norm{\cdot}_2$ and $w_n \longrightarrow w$ w.r.t.~$\norm{\cdot}_{\infty}$. Then 
\begin{align} \label{eq:p^w(f)-contin-in-w,f}
p_{\sigma}^{w_n}(f_n) \underset{\norm{\cdot}_2}{\longrightarrow} p_{\sigma}^w(f) \qquad (n\to\infty).
\end{align}
\end{lm}

\begin{proof}
Clearly, there is also a constant $\ol{c} \in (0,\infty)$ such that $\ul{c} \le w_n, w \le \ol{c}$ for all $n \in \N$ and, moreover,
\begin{align} \label{eq:p^w(f)-contin-in-w,f-1}
\norm{p_{\sigma}^{w_n}(f_n)-p_{\sigma}^{w}(f)}_2 
\le \norm{p_{\sigma}^{w_n}(f_n)-p_{\sigma}^{w_n}(f)}_2 + \norm{p_{\sigma}^{w_n}(f)-p_{\sigma}^{w}(f)}_2
\end{align}
for all $n\in\N$. 
It is easy to show that the first term on the right-hand side of~\eqref{eq:p^w(f)-contin-in-w,f-1} converges to $0$. %as $n\to\infty$.
Indeed, by the $\norm{\cdot}_{2,w_n}$-contractivity of $p_{\sigma}^{w_n}$ (Proposition~\ref{prop:properties-of-p}), we have
\begin{align} \label{eq:p^w(f)-contin-in-w,f-2}
\norm{p_{\sigma}^{w_n}(f_n)-p_{\sigma}^{w_n}(f)}_2^2 \le (1/\ul{c}) \norm{p_{\sigma}^{w_n}(f_n)-p_{\sigma}^{w_n}(f)}_{2,w_n}^2 \le (\ol{c}/\ul{c}) \norm{f_n-f}_2^2
\end{align}
for every $n\in\N$.
It thus remains to show that the second term on the right-hand side of~\eqref{eq:p^w(f)-contin-in-w,f-1} converges to $0$, too. 
In order to do so, we have only to show that $(p_{\sigma}^{w_n}(f))$ is a minimizing sequence for $J_{f,w}|_{L^2_{\sigma}(Q,\R)}$ (Theorem~\ref{thm:ex-and-uniqueness-of-p}). Clearly, %$p_n = p_{\sigma}^{w_n}(f) \in L^2(Q,\R)$ 
\begin{align} \label{eq:p^w(f)-contin-in-w,f-3}
p_n := p_{\sigma}^{w_n}(f) \in L^2_{\sigma}(Q,\R) %\qquad (n\in\N)
\qquad \text{and} \qquad
f^* := p_{\sigma}^w(f) \in L^2_{\sigma}(Q,\R)
\end{align}
for all $n\in\N$ and therefore we have %, writing $f^* := p_{\sigma}^w(f)$ as usual, 
\begin{align} \label{eq:p^w(f)-contin-in-w,f-4}
J_{f,w}(f^*) 
&= \inf_{g\in L^2_{\sigma}(Q,\R)} J_{f,w}(g) \le J_{f,w}(p_n) %= J_{f,w_n}(p_n) + \int_Q |p_n-f|^2 (w-w_n)\d\lambda
= \inf_{g\in L^2_{\sigma}(Q,\R)} J_{f,w_n}(g) + \int_Q |p_n-f|^2 (w-w_n)\d\lambda \notag \\
&\le J_{f,w_n}(f^*) + \int_Q |p_n-f|^2 (w-w_n)\d\lambda \notag \\
&= J_{f,w}(f^*) + \int_Q|f^*-f|^2(w_n-w)\d\lambda + \int_Q |p_n-f|^2 (w-w_n)\d\lambda
\end{align}
for every $n\in\N$. Since $\norm{w_n-w}_{\infty} \longrightarrow 0$ as $n\to\infty$ %$w_n \longrightarrow w$ w.r.t.~$\norm{\cdot}_{\infty}$ 
and since
\begin{align} \label{eq:p^w(f)-contin-in-w,f-5}
\ul{c} \norm{p_n-f}_2^2 \le  J_{f,w_n}(p_n) = \inf_{g\in L^2_{\sigma}(Q,\R)} J_{f,w_n}(g) \le J_{f,w_n}(0) \le \ol{c} \norm{f}_2^2 
\qquad (n\in\N),
\end{align}
the two integrals on the right-hand side of~\eqref{eq:p^w(f)-contin-in-w,f-4} converge to $0$. Consequently, the lower and upper bound for $J_{f,w}(p_n)$ from~\eqref{eq:p^w(f)-contin-in-w,f-4} are asymptotically equal as $n\to\infty$ and thus
\begin{align} \label{eq:p^w(f)-contin-in-w,f-6}
J_{f,w}(p_n) \longrightarrow J_{f,w}(f^*) \qquad (n\to\infty). 
\end{align}
In view of~(\ref{eq:p^w(f)-contin-in-w,f-3}.a) and~\eqref{eq:p^w(f)-contin-in-w,f-6}, it is now clear that $(p_{\sigma}^{w_n}(f)) = (p_n)$ is a minimizing sequence for $J_{f,w}|_{L^2_{\sigma}(Q,\R)}$, as desired. 
\end{proof}

\begin{thm} \label{thm:approx-thm-2-norm}
Suppose $f \in L^2(Q,\R)$ and $w \in S^{\infty}(Q,(0,\infty))$. Then there exist dyadic grids $G_n$ on $Q$ and $G_n$-constant functions $f_n:Q\to\R$ and $w_n:Q\to(0,\infty)$ such that $f_n \longrightarrow f$ w.r.t.~$\norm{\cdot}_2$ and $w_n \longrightarrow w$ w.r.t.~$\norm{\cdot}_{\infty}$ and 
\begin{align} \label{eq:approx-thm-2-norm-assertion-1}
p_{\sigma}^{w_n}(f_n) \underset{\norm{\cdot}_2}{\longrightarrow} p_{\sigma}^w(f) \qquad (n\to\infty).
\end{align}
In particular, $S_{\sigma}(Q,\R)$ is dense w.r.t.~$\norm{\cdot}_2$ in $L^2_{\sigma}(Q,\R)$. Additionally, in the special case where $w$ itself is already dyadically grid-constant, one has the following estimate on the convergence rate:
\begin{align} \label{eq:approx-thm-2-norm-assertion-2}
\norm{p_{\sigma}^w(f_n)-p_{\sigma}^w(f)}_2 \le C \norm{f_n-f}_2,
\end{align}
where $C := (\ol{c}/\ul{c})^{1/2}$ and $\ul{c} := \essinf_{x \in Q} w(x)$ and $\ol{c} := \esssup_{x\in Q} w(x)$.
\end{thm}

\begin{proof}
In view of the density of $S(Q,\R)$ in $L^2(Q,\R)$ (Lemma~\ref{lm:S-dense-in-L^2}) and the definition of $S^{\infty}(Q,(0,\infty))$, there exist dyadic grids $G_n^f$, $G_n^w$ and $G_n^f$-constant functions $f_n:Q\to\R$ as well as $G_n^w$-constant functions $w_n:Q\to (0,\infty)$ such that
\begin{align} \label{eq:approx-thm-2-norm-1}
f_n \underset{\norm{\cdot}_{2}}{\longrightarrow} f \qquad (n\to\infty)
\qquad \text{and} \qquad
w_n \underset{\norm{\cdot}_{\infty}}{\longrightarrow} w \qquad (n\to\infty). 
\end{align}
Since the grids $G_n^f$, $G_n^w$ are dyadic, their union $G_n := G_n^f \cup G_n^w$ is just the finer of the two grids and therefore $G_n$ is dyadic as well and the functions $f_n, w_n$ are both $G_n$-constant. So, by~\eqref{eq:approx-thm-2-norm-1} and Lemma~\ref{lm:p^w(f)-contin-in-w,f}, the desired convergence~\eqref{eq:approx-thm-2-norm-assertion-1} follows. 
\smallskip

Also, the asserted estimate~\eqref{eq:approx-thm-2-norm-assertion-2} on the convergence rate immediately follows by the $\norm{\cdot}_{2,w}$-contractivity of $p_{\sigma}^w$ (Proposition~\ref{prop:properties-of-p}). 
\smallskip

It remains to prove the density of $S_{\sigma}(Q,\R)$ in $L^2_{\sigma}(Q,\R)$. So, let $f \in L^2_{\sigma}(Q,\R)$ and, moreover,  choose $f_n \in S(Q,\R)$ with $f_n \longrightarrow f$ w.r.t.~$\norm{\cdot}_2$ (Lemma~\ref{lm:S-dense-in-L^2}) and $w_0 := 1$. We then have $f = p_{\sigma}^{w_0}(f)$ and $p_{\sigma}^{w_0}(f_n) \in S_{\sigma}(Q,\R)$ for every $n \in \N$ by Theorem~\ref{thm:spec-case-grid-const-fct}. Combining these two facts with~\eqref{eq:approx-thm-2-norm-assertion-2}, we see that $f$ lies in the $\norm{\cdot}_{2}$-closure of $S_{\sigma}(Q,\R)$, as desired. 
\end{proof}

\subsection{Approximation in the $\infty$-norm}

In this subsection, we prove our approximation result w.r.t.~the $\infty$-norm. We need two lemmas for that purpose. 

\begin{lm} \label{lm:difference-of-inf-sup-expressions}
If $\phi, \psi:S_1\times S_2 \to \R$ are bounded functions on arbitrary sets $S_1, S_2$, then 
\begin{align} \label{eq:difference-of-inf-sup-expressions}
\Big| \inf_{u\in S_1} \Big( \sup_{v \in S_2} \phi(u,v) \Big) &- \inf_{u\in S_1} \Big( \sup_{v \in S_2} \psi(u,v) \Big) \Big|
, 
\Big| \sup_{u\in S_1} \Big( \inf_{v \in S_2} \phi(u,v) \Big) - \sup_{u\in S_1} \Big( \inf_{v \in S_2} \psi(u,v) \Big) \Big| \notag \\
&\le \sup_{(u,v)\in S_1\times S_2} \big| \phi(u,v)-\psi(u,v)\big|.
\end{align}
\end{lm}

\begin{proof}
%A straightforward argument, using just the definition of suprema, yields the following fact:
We begin with the following elementary observation (which follows directly from the definition of suprema): if $\alpha, \beta: S \to \R$ are bounded functions on any set $S$, then
\begin{align}
\big| \sup_{s\in S} \alpha(s) - \sup_{s\in S} \beta(s) \big| 
\le \sup_{s\in S} |\alpha(s)-\beta(s)|.
\end{align}  
Applying this observation twice, in conjunction with the fact that %$\inf_{s\in S} \gamma(s) = -\sup_{s\in S} (-\gamma(s))$ 
\begin{align}
\inf_{s\in S} \gamma(s) = -\sup_{s\in S} (-\gamma(s))
\end{align}
for any function $\gamma: S \to \R$, we immediately obtain the desired estimates~\eqref{eq:difference-of-inf-sup-expressions}.   
\end{proof}

\begin{lm} \label{lm:Av_f,w-bounded-and-equicont}
Suppose $f_n,f \in L^{\infty}(Q,\R)$ and $w_n, w \in L^{\infty}(Q,[\ul{c},\infty))$ for some $\ul{c} \in (0,\infty)$ such that $f_n \longrightarrow f$ w.r.t.~$\norm{\cdot}_{\infty}$ and $w_n \longrightarrow w$ w.r.t.~$\norm{\cdot}_{\infty}$. Then $\mathcal{Q}_Q^{>} \ni E \mapsto \Av_{f_n,w_n}(E), \Av_{f,w}(E)$ are bounded functions and 
\begin{align} \label{eq:Av_f,w-bounded-and-equicont}
\sup_{E \in \mathcal{Q}_Q^{>}} \big| \Av_{f_n,w_n}(E) - \Av_{f,w}(E) \big| \longrightarrow 0 \qquad (n\to\infty).
\end{align} 
%where $\mathcal{Q}_Q^{>} := \{E \in \mathcal{Q}_Q: \lambda(E) > 0\}$. 
\end{lm}

\begin{proof}
Clearly, there is also a constant $\ol{c} \in (0,\infty)$ such that $\ul{c} \le w_n, w \le \ol{c}$ for all $n \in \N$ and, moreover,
\begin{align}
\Av_{f_n,w_n}(E) = \frac{\phi_{f_n,w_n}(E)}{\psi_{w_n}(E)}
\qquad \text{and} \qquad
\Av_{f,w}(E) = \frac{\phi_{f,w}(E)}{\psi_{w}(E)}
\end{align}
for $E \in \mathcal{Q}_Q^{>}$, where $\phi_{f',w'}(E)$, $\psi_{w'}(E)$ for arbitrary $f' \in L^{\infty}(Q,\R)$, $w' \in L^{\infty}(Q,[\ul{c},\infty))$ are defined by
\begin{align}
\phi_{f',w'}(E) := \frac{1}{\lambda(E)}\int_E f'w'\d\lambda
\qquad \text{and} \qquad
\psi_{w'}(E) := \frac{1}{\lambda(E)}\int_E w'\d\lambda.
\end{align}
Since for all $f' \in L^{\infty}(Q,\R)$, $w' \in L^{\infty}(Q,[\ul{c},\infty))$ and $E \in  \mathcal{Q}_Q^{>}$ 
\begin{align}
\phi_{f',w'}(E) \le \norm{f'}_{\infty} \norm{w'}_{\infty} 
\qquad \text{and} \qquad
\psi_{w'}(E) \ge \ul{c}, 
\end{align}
we see that $\mathcal{Q}_Q^{>} \ni E \mapsto \Av_{f_n,w_n}(E), \Av_{f,w}(E)$ are bounded functions. 
Since, moreover, 
\begin{gather}
\big| \phi_{f_n,w_n}(E)-\phi_{f,w}(E) \big| \le \norm{f_n w_n - f w} _{\infty} \le \norm{f_n-f}_{\infty} \ol{c} + \norm{f}_{\infty} \norm{w_n-w}_{\infty}, \\
\big| \psi_{w_n}(E)-\psi_{w}(E) \big| \le \norm{w_n-w}_{\infty},
\end{gather}
we also see that
\begin{align}
&\big| \Av_{f_n,w_n}(E) - \Av_{f,w}(E) \big| \notag \\
&\qquad \qquad \le \frac{1}{\psi_{w_n}(E)} \big| \phi_{f_n,w_n}(E)-\phi_{f,w}(E) \big| + \bigg| \frac{1}{\psi_{w_n}(E)} - \frac{1}{\psi_{w_n}(E)} \bigg| \big| \phi_{f,w}(E) \big| \notag \\
&\qquad \qquad \le (\ol{c}/\ul{c}) \norm{f_n-f}_{\infty} + \norm{f}_{\infty} (1/\ul{c}+\ol{c}/\ul{c}^2) \norm{w_n-w}_{\infty}
\end{align}
for all $E \in \mathcal{Q}_Q^{>}$ and $n\in\N$, which yields the desired uniform convergence~\eqref{eq:Av_f,w-bounded-and-equicont}. 
\end{proof}

\begin{thm} \label{thm:approx-thm-infty-norm}
Suppose $f \in S^{\infty}(Q,\R)$ and $w \in S^{\infty}(Q,(0,\infty))$. Then there exist dyadic grids $G_n$ on $Q$ and $G_n$-constant functions $f_n:Q\to\R$ and $w_n:Q\to(0,\infty)$ such that $f_n \longrightarrow f$ w.r.t.~$\norm{\cdot}_{\infty}$ and $w_n \longrightarrow w$ w.r.t.~$\norm{\cdot}_{\infty}$ and 
\begin{align} \label{eq:approx-thm-infty-norm-assertion-1}
p_{\sigma}^{w_n}(f_n) \underset{\norm{\cdot}_{\infty}}{\longrightarrow} p_{\sigma}^w(f) \qquad (n\to\infty).
\end{align}
In particular, $S_{\sigma}(Q,\R)$ is dense w.r.t.~$\norm{\cdot}_{\infty}$ in $S^{\infty}_{\sigma}(Q,\R)$. Additionally, in the special case where $w$ itself is already dyadically grid-constant, one has the following estimate on the convergence rate:
\begin{align} \label{eq:approx-thm-infty-norm-assertion-2}
\norm{p_{\sigma}^w(f_n)-p_{\sigma}^w(f)}_{\infty} \le \norm{f_n-f}_{\infty}.
\end{align}
\end{thm}

\begin{proof}
In view of the definitions of $S^{\infty}(Q,\R)$ and $S^{\infty}(Q,(0,\infty))$, there exist dyadic grids $G_n^f$, $G_n^w$ and $G_n^f$-constant functions $f_n:Q\to\R$ as well as $G_n^w$-constant functions $w_n:Q\to (0,\infty)$ such that
\begin{align} \label{eq:approx-thm-infty-norm-1}
f_n \underset{\norm{\cdot}_{\infty}}{\longrightarrow} f \qquad (n\to\infty)
\qquad \text{and} \qquad
w_n \underset{\norm{\cdot}_{\infty}}{\longrightarrow} w \qquad (n\to\infty). 
\end{align}
Since the grids $G_n^f$, $G_n^w$ are dyadic, the union $G_n := \bigcup_{m=1}^n G_m^f \cup G_m^w$ %is just the finest of the $2n$ grids in the union so that $G_n$ is dyadic as well and $f_n, w_n$ are also $G_n$-constant. 
is dyadic as well and the functions $f_n, w_n$ are also $G_n$-constant ($G_n$ is just the finest of the $2n$ grids in the union). In fact, we see that all the functions $f_m,w_m$ with $m \le n$ are $G_n$-constant. So, by Corollary~\ref{cor:avg-G}, $a_{\sigma,\#}^{w_m,G_n}(f_m)$ is a representative of $p_{\sigma}^{w_m}(f_m)$ for every $m \le n$ and $a_{\sigma,\#}^{w_n,G_n}(f_n)$ is a representative of $p_{\sigma}^{w_n}(f_n)$ with $\# \in \{\mathrm{is},\mathrm{si}\}$ and, thus,
\begin{align}  \label{eq:approx-thm-infty-norm-2}
\norm{p_{\sigma}^{w_n}(f_n) - p_{\sigma}^{w_m}(f_m)}_{\infty} 
&\le \sup_{x\in Q} \big| a_{\sigma,\#}^{w_n,G_n}(f_n)(x) - a_{\sigma,\#}^{w_m,G_n}(f_m)(x) \big| \notag \\
&\le \sup_{x\in Q} \sup_{(L,U) \in \mathcal{L}_{\sigma}^{G_n}(x) \times \mathcal{U}_{\sigma}^{G_n}(x)} \big| \Av_{f_n,w_n}(L\cap U) -  \Av_{f_m,w_m}(L\cap U) \big| \notag \\
&\le \sup_{E\in \mathcal{Q}_Q^{>}} \big| \Av_{f_n,w_n}(E) -  \Av_{f_m,w_m}(E) \big|
\qquad (m \le n)
\end{align}
by virtue of Lemma~\ref{lm:difference-of-inf-sup-expressions} and the boundedness part of Lemma~\ref{lm:Av_f,w-bounded-and-equicont}. In view of~\eqref{eq:approx-thm-infty-norm-1}, \eqref{eq:approx-thm-infty-norm-2} and the convergence part of Lemma~\ref{lm:Av_f,w-bounded-and-equicont}, we now see that $(p_{\sigma}^{w_n}(f_n))$ is a Cauchy sequence w.r.t.~$\norm{\cdot}_{\infty}$ and hence converges to some $f^* \in L^{\infty}(Q,\R)$ w.r.t.~$\norm{\cdot}_{\infty}$. Since $(p_{\sigma}^{w_n}(f_n))$, by~\eqref{eq:approx-thm-infty-norm-1} and Lemma~\ref{lm:p^w(f)-contin-in-w,f}, also converges to $p_{\sigma}^w(f)$ w.r.t.~the weaker norm $\norm{\cdot}_2$, the limits must coincide, which proves the desired convergence~\eqref{eq:approx-thm-infty-norm-assertion-1}. 
\smallskip

Also, the asserted estimate~\eqref{eq:approx-thm-infty-norm-assertion-2} on the convergence rate immediately follows by the $\norm{\cdot}_{\infty}$-contractivity of $p_{\sigma}^w$ (Proposition~\ref{prop:properties-of-p}). 
\smallskip

It remains to prove the density of $S_{\sigma}(Q,\R)$ in $S^{\infty}_{\sigma}(Q,\R)$. So, let $f \in S^{\infty}_{\sigma}(Q,\R)$ and, moreover,  choose $f_n \in S(Q,\R)$ with $f_n \longrightarrow f$ w.r.t.~$\norm{\cdot}_{\infty}$ and $w_0 := 1$. We then have, in particular, $f \in L^2_{\sigma}(Q,\R)$ by the definition of $S^{\infty}_{\sigma}(Q,\R)$ and therefore $f = p_{\sigma}^{w_0}(f)$. We also have $p_{\sigma}^{w_0}(f_n) \in S_{\sigma}(Q,\R)$ for every $n \in \N$ by Theorem~\ref{thm:spec-case-grid-const-fct}. Combining these facts with~\eqref{eq:approx-thm-infty-norm-assertion-2}, we see that $f$ lies in the $\norm{\cdot}_{\infty}$-closure of $S_{\sigma}(Q,\R)$, as desired.  
\end{proof}

\section{Characterization of monotonic regression functions}%{Characterization of monotonic regressors}

In this section, we deal with generalized monotonic regression problems, which naturally arise in maximum smoothed likelihood estimation~\cite{GrJo10}. We will show that for a given square-integrable function $f$ and a weight function $w$, the solution of a wide class of generalized monotonic regression problems is nothing but the solution $p_{\sigma}^w(f)$ of the standard monotonic regression problem treated so far. 
\smallskip

We begin by extending a well-known identity from the discrete case (Theorem~1.3.6 of~\cite{RoWrDy}) to non-discrete situations. %to non-discrete functions. %to the non-discrete case.  

\begin{lm} \label{lm:phi(f^*)}
Suppose $f \in L^2(Q,\R)$, $w \in S^{\infty}(Q,(0,\infty))$ and let $f^* := p_{\sigma}^w(f)$. Then
\begin{align} \label{eq:phi(f^*)}
\scprd{f-f^*,\phi\circ f^*}_{2,w} = 0
%\int_Q (f-f^*) \cdot (\phi\circ f^*) \, w \d\lambda = 0
\end{align}
for every function $\phi:\R\to\R$ of bounded variation. 
\end{lm}

\begin{proof}
We proceed in four steps.
As a first step, we prove the assertion for bounded continuous functions $\phi$. 
So, let $\phi$ be bounded and continuous. Choose dyadic grids $G_n$ and $G_n$-constant functions $f_n: Q \to \R$ and $w_n:Q \to (0,\infty)$ according to Theorem~\ref{thm:approx-thm-2-norm} and write $f_n^* := p_{\sigma}^{w_n}(f_n)$. We then have, by~\eqref{eq:approx-thm-2-norm-assertion-1} of that theorem, that
\begin{align} \label{eq:phi(f^*)-1}
f_n^* \underset{\norm{\cdot}_2}{\longrightarrow} f^* \qquad (n\to\infty)
\end{align}
and therefore, by the assumed continuity and boundedness of $\phi$ and the dominated convergence theorem, %also
\begin{align} \label{eq:phi(f^*)-2}
\phi\circ f_n^* \underset{\norm{\cdot}_2}{\longrightarrow} \phi\circ f^* \qquad (n\to\infty).
\end{align}
Also, by virtue of Theorem~\ref{thm:spec-case-grid-const-fct}, we have that
\begin{align} \label{eq:phi(f^*)-3}
f_n^* = \sum_{x\in G} f_{n|}^*(x) \chi_{C^{G_n}(x)} \qquad (n\in\N),
\end{align}
where $f_{n|}^* := p_{\sigma}^{w_{n|}}(f_{n|})$ with $f_{n|}:= f_n|_{G_n}$ and $w_{n|}:= w_n|_{G_n}$. 
So, on the one hand, it follows by~\eqref{eq:phi(f^*)-3} and the $G_n$-constancy of $f_n$, $f_n^*$, $w_n$ and the equidistance of $G_n$ that
\begin{align} \label{eq:phi(f^*)-4}
\scprd{f_n-f_n^*,\phi\circ f_n^*}_{2,w_n} = \sum_{x\in G_n} \big(f_{n|}(x)-f_{n|}^*(x)\big) \phi(f_{n|}^*(x)) w_n(x) \cdot \operatorname{vol}_{G_n} = 0
\end{align}
for all $n\in\N$, where in the second equality we used the discrete analog  of the lemma (Theorem~1.3.6 in~\cite{RoWrDy}). And, on the other hand, it follows by the choice of $f_n$, $w_n$ and by~\eqref{eq:phi(f^*)-1}, \eqref{eq:phi(f^*)-2} that
\begin{align} \label{eq:phi(f^*)-5}
\scprd{f_n-f_n^*,\phi\circ f_n^*}_{2,w_n} 
%= \scprd{f_n-f_n^*,(\phi\circ f_n^*)\, w_n}_{2} 
\longrightarrow  \scprd{f-f^*,\phi\circ f^*}_{2,w} \qquad (n\to\infty).
\end{align}
Combining~\eqref{eq:phi(f^*)-4} and~\eqref{eq:phi(f^*)-5}, we obtain the desired equation~\eqref{eq:phi(f^*)} for bounded continuous functions $\phi$. %the special class of functions from this first step. 
\smallskip

As a second step, we prove the assertion for functions $\phi$ of the form $\chi_{[c,\infty)}$ or $\chi_{(c,\infty)}$ with $c \in \R$. 
So, let $\phi = \chi_{[c,\infty)}$ or $\phi = \chi_{(c,\infty)}$ with some $c \in \R$. We can then, of course, find bounded continuous functions $\phi_n: \R \to [0,1]$ (ramp functions with steeper and steeper ramps) such that
\begin{align} \label{eq:phi(f^*)-6}
\phi_n(u) \longrightarrow \phi(u) \qquad (n\to\infty) 
\qquad \text{and} \qquad 
|\phi_n(u)| \le 1 \qquad (n\in\N)
\end{align}
for every $u \in \R$. Consequently, the desired equation~\eqref{eq:phi(f^*)} follows by the dominated convergence theorem and the first step.
\smallskip

As a third step, we prove the assertion for bounded monotonically increasing functions $\phi$. 
So, let $\phi$ be bounded and monotonically increasing. We can then decompose $\phi$ into a bounded continuous part $\phi_{\mathrm{c}}: \R \to\R$ and a purely discrete part:
\begin{align} \label{eq:phi(f^*)-7}
\phi(u) = \phi_{\mathrm{c}}(u) + \sum_{k\in K} h_k^- \chi_{[c_k,\infty)}(u) + \sum_{k\in K} h_k^+ \chi_{(c_k,\infty)}(u) 
\qquad (u\in\R),
\end{align} 
where $(c_k)_{k\in K}$ is an enumeration of the countably many discontinuities of $\phi$, %and $h_k^- := \phi(c_k)-\phi(c_k-)$  and $h_k^+ := \phi(c_k+)-\phi(c_k)$.
\begin{align*}
h_k^- := \phi(c_k)-\phi(c_k-) \qquad \text{and} \qquad h_k^+ := \phi(c_k+)-\phi(c_k),
\end{align*}
and the sums in~\eqref{eq:phi(f^*)-7} are uniformly convergent (Lemma~1.6.3 of~\cite{Ta}). Consequently, the desired equation~\eqref{eq:phi(f^*)} follows by the first step and second step.
\smallskip

As a fourth step, we finally exploit that every function $\phi: \R\to \R$ of bounded variation can be written as the difference  $\phi = \phi_1 - \phi_2$ of two bounded monotonically increasing functions $\phi_1, \phi_2$ (Proposition~4.4.2 of~\cite{Co}). And therefore, the general assertion of the lemma follows by the third step.
\end{proof}

\begin{cor} \label{cor:Av_f,w(f^*=c)=c}
Suppose $f \in L^2(Q,\R)$, $w \in S^{\infty}(Q,(0,\infty))$ and let $f^* := p_{\sigma}^w(f)$. Then
\begin{align} \label{eq:Av_f,w(f^*=c)=c-assertion-1}
\Av_{f,w}(L \cap \{f^*\ge c\}) \ge c
\qquad \text{and} \qquad
\Av_{f,w}(\{f^*\le c\} \cap U) \le c 
\end{align}
for every $c \in \R$ and every $L \in \mathcal{L}_{\sigma}^{\mathrm{mb}} := \mathcal{L}_{\sigma} \cap \mathcal{Q}_Q$ and $U \in \mathcal{U}_{\sigma}^{\mathrm{mb}} := \mathcal{U}_{\sigma} \cap \mathcal{Q}_Q$ such that $L \cap \{f^*\ge c\}$ and $\{f^*\le c\} \cap U$ are non-null sets. In particular, 
\begin{align} \label{eq:Av_f,w(f^*=c)=c-assertion-2}
\Av_{f,w}(\{f^*=c\}) = c
\end{align}
for every $c \in \R$ for which $\{f^*=c\}$ is a non-null set. Additionally, %if $f > a$ or $f < b$ almost everywhere, then also $f^* > a$ or $f^* < b$ almost everywhere, respectively. 
if $I \subset \R$ is any interval with $f(x) \in I$ for a.e.~$x\in Q$, then one also has $f^*(x) \in I$ for a.e.~$x \in Q$. 
\end{cor}

\begin{proof}
We proceed in four steps. 
As a first step, we prove~(\ref{eq:Av_f,w(f^*=c)=c-assertion-1}.a). 
So, let $c \in \R$ and $L \in \mathcal{L}_{\sigma}^{\mathrm{mb}}$. We then have
\begin{align} \label{eq:Av_f,w(f^*=c)=c-1}
&\int_{L\cap\{f^*\ge c\}} fw \d\lambda - c \int_{L\cap\{f^*\ge c\}} w \d\lambda \ge \int_{L\cap\{f^*\ge c\}} (f-f^*)w \d\lambda \notag \\
&\qquad \qquad = \int_{\{f^*\ge c\}} (f-f^*)w \d\lambda - \int_{Q\setminus L \cap \{f^*\ge c\}} (f-f^*)w \d\lambda \notag \\
&\qquad \qquad = \scprd{f-f^*,\chi_{[c,\infty)}\circ f^*}_{2,w} - \scprd{f-f^*,\chi_{Q\setminus L} \cdot \chi_{[c,\infty)} \circ f^*}_{2,w}.
\end{align}
Since $\chi_{[c,\infty)}$ is of bounded variation, the first scalar product on the right-hand side of~\eqref{eq:Av_f,w(f^*=c)=c-1} is equal to $0$ by Lemma~\ref{lm:phi(f^*)}. Since, moreover, $\chi_{Q\setminus L}$ is $\sigma$-monotonic by Lemma~\ref{lm:char-isotonic-fcts} and since $f^* \in L^2_{\sigma}(Q,\R)$ and $\chi_{[c,\infty)}$ is monotonically increasing, it follows that 
\begin{align*}
\chi_{Q\setminus L} \cdot \chi_{[c,\infty)} \circ f^* \in L^2_{\sigma}(Q,\R)
\end{align*}
and therefore the second scalar product on the right-hand side of~\eqref{eq:Av_f,w(f^*=c)=c-1} is less than or equal to $0$ by Proposition~\ref{prop:geometr-char-of-p}. Consequently, 
\begin{align} \label{eq:Av_f,w(f^*=c)=c-2}
\int_{L\cap\{f^*\ge c\}} fw \d\lambda - c \int_{L\cap\{f^*\ge c\}} w \d\lambda \ge 0,
\end{align}
and therefore~(\ref{eq:Av_f,w(f^*=c)=c-assertion-1}.a) holds true whenever $L\cap\{f^*\ge c\}$ is non-null, as desired. 
\smallskip

As a second step, we prove~(\ref{eq:Av_f,w(f^*=c)=c-assertion-1}.b). 
So, let $c \in \R$ and $U \in \mathcal{U}_{\sigma}^{\mathrm{mb}}$ such that $\{f^*\le c\} \cap U$ is non-null. We then have
\begin{align} \label{eq:Av_f,w(f^*=c)=c-3}
-f^* = -p_{\sigma}^w(f) = p_{-\sigma}^w(-f) =:(-f)^*
\qquad \text{and} \qquad
U \in \mathcal{U}_{\sigma}^{\mathrm{mb}} = \mathcal{L}_{-\sigma}^{\mathrm{mb}}
\end{align}
by virtue of Lemma~\ref{lm:minus-sigma} and therefore
\begin{align} \label{eq:Av_f,w(f^*=c)=c-4}
\Av_{f,w}(\{f^*\le c\} \cap U) &= \Av_{f,w}(\{(-f)^*\ge -c\} \cap U) \notag \\
&= -\Av_{-f,w}(U \cap \{(-f)^*\ge -c\}) \le c
\end{align}
by virtue of the first step applied to $-f, w$ and $-\sigma$, as desired. 
\smallskip

As a third step, we prove~\eqref{eq:Av_f,w(f^*=c)=c-assertion-2}. Indeed, applying Lemma~\ref{lm:phi(f^*)} with $\phi := \chi_{\{c\}}$, we get
\begin{align} \label{eq:Av_f,w(f^*=c)=c-5}
\int_{\{f^*=c\}} f-c \d\lambda = \int_{\{f^*=c\}} (f-f^*)\cdot \chi_{\{c\}}\circ f^* \d\lambda = 0
\end{align}
for every $c \in \R$. And from this, in turn, \eqref{eq:Av_f,w(f^*=c)=c-assertion-2} is obvious. (Alternatively, we could also observe that~\eqref{eq:Av_f,w(f^*=c)=c-assertion-1} extends to $L \in \ol{ \mathcal{L} }_{\sigma}^{\mathrm{mb}}$ and $U \in \ol{ \mathcal{U} }_{\sigma}^{\mathrm{mb}}$ and apply this extended relation to $L := \{f^* \le c\}$ and $U := \{f^* \ge c\}$ which do belong, respectively, to $\ol{ \mathcal{L} }_{\sigma}^{\mathrm{mb}}$ and $\ol{ \mathcal{U} }_{\sigma}^{\mathrm{mb}}$ by the remarks around~\eqref{eq:extend-upper-set-on-X-minus-N-to-upper-set-on-X}.) %Um die letzten beiden Elementrelationen zu sehen brauche aber auch noch, dass jede \sigma-untere /-obere Menge auf Q\setminus N enthalten ist in einer \sigma-unteren/-oberen Menge auf Q!)
\smallskip

As a fourth step, we prove the remaining interval-preservation property. What we have to show for that purpose is that the projection $p_{\sigma}^w$ preserves all strict and all non-strict inequalities (in the a.e.~sense) of the forms 
\begin{align} \label{eq:Av_f,w(f^*=c)=c-6}
f > c, \qquad f < c, \qquad f \ge c, \qquad f \le c 
\end{align}
between $f$ and a constant $c \in \R$, respectively. In view of~Proposition~\ref{prop:elem-properties-of-p}(ii) and~\ref{prop:properties-of-p}(ii), the preservation of the non-strict inequalities in~\eqref{eq:Av_f,w(f^*=c)=c-6} is clear. And in view of~(\ref{eq:Av_f,w(f^*=c)=c-4}.a), it is sufficient to prove the preservation of just one of the strict inequalities in~\eqref{eq:Av_f,w(f^*=c)=c-6}. So, let $c \in \R$ and %$f(x) > c$ for a.e.~$x\in Q$.
\begin{align} \label{eq:Av_f,w(f^*=c)=c-7}
f(x) > c \text{ \,for a.e.~}x\in Q 
\end{align}
We then have at least $f^* = p_{\sigma}^w(f) \ge p_{\sigma}^w(c) = c$ by the preservation of the non-strict inequalities in~\eqref{eq:Av_f,w(f^*=c)=c-6}. It thus remains to show that $\{f^* = c\}$ is a null set -- but this is an immediate consequence of~\eqref{eq:Av_f,w(f^*=c)=c-5} and~\eqref{eq:Av_f,w(f^*=c)=c-7}.
\end{proof}

\begin{lm} \label{lm:appr-by-bounded-mon-phi_n}
If $\phi: I \to \R \cup \{\pm\infty\}$ is monotonically increasing on an interval $I \subset \R$ and finite on the interior of $I$, %such that $\phi(I^{\circ}) \subset \R$. 
then there exist bounded monotonically increasing functions $\phi_n:\R \to \R$ such that
\begin{align} \label{eq:appr-by-bounded-mon-phi_n}
\phi_n \circ g \underset{\norm{\cdot}_2}\longrightarrow \phi \circ \phi \qquad (n\to \infty)
\end{align}
for every function $g: Q\to I$ with $\phi \circ g \in L^2(Q,\R)$.
\end{lm}

\begin{proof}
Write $\alpha := \inf I \in \R \cup \{-\infty\}$ and $\beta := \sup I \in \R \cup \{\infty\}$ and choose $\alpha_n, \beta_n \in (\alpha,\beta) = \operatorname{int} I$ with
\begin{align} \label{eq:appr-by-bounded-mon-phi_n-1}
\alpha_n \searrow \alpha \qquad (n\to\infty) 
\qquad \text{and} \qquad
\beta_n \nearrow \beta \qquad (n\to\infty).
\end{align}
We then define the functions $\phi_n$ by
\begin{align} \label{eq:appr-by-bounded-mon-phi_n-2}
\phi_n(u) := \phi(\alpha_n) \chi_{(-\infty,\alpha_n)}(u) + \phi(u) \chi_{[\alpha_n,\beta_n]}(u) + \phi(\beta_n) \chi_{(\beta_n,\infty)}(u)
\qquad (u \in\R).
\end{align}
It is straightforward to verify that $\phi_n$ is bounded and monotonically increasing and it remains to establish~\eqref{eq:appr-by-bounded-mon-phi_n}. So, let $g: Q\to I$ be a function with $\phi \circ g \in L^2(Q,\R)$. It then follows directly from the definition~\eqref{eq:appr-by-bounded-mon-phi_n-2} that
\begin{align} \label{eq:appr-by-bounded-mon-phi_n-3}
\norm{\phi_n\circ g - \phi\circ g}_2^2 = \int_Q h_n^-(x) \d x + \int_Q h_n^+(x)\d x
\end{align}
where $h_n^-(x) := |\phi(\alpha_n)-\phi(g(x))|^2 \chi_{\{g<\alpha_n\}}(x)$ and $h_n^+(x) := |\phi(\beta_n)-\phi(g(x))|^2 \chi_{\{g>\beta_n\}}(x)$ for $x \in Q$ and $n \in \N$. Since $\phi$ and $(\alpha_n)$, $(\beta_n)$ are monotonic, we easily see that
\begin{align} \label{eq:appr-by-bounded-mon-phi_n-4}
h_n^{\pm}(x) \longrightarrow 0 \qquad (n\to\infty)
\end{align}
for every $x \in Q$ and that $(h_n^{\pm})$ is monotonically decreasing and thus
\begin{align}
0 \le h_{n+1}^{\pm} \le h_n^{\pm} \le h_1^{\pm} \qquad (n\in\N).
\end{align}
Since $h_1^{\pm}$ is integrable by our assumption on $g$, the right-hand side of~\eqref{eq:appr-by-bounded-mon-phi_n-3} converges to $0$ as $n\to\infty$ by the dominated convergence theorem, as desired. 
\end{proof}

With the above lemmas at hand, we can now establish the main result of the section. It is a generalization of a result from~\cite{GrJo10} (Theorem~1), where the case of univariate and continuous functions $f$ and $w$ is considered. We proceed in a very %fairly/quite 
different way than~\cite{GrJo10}. %because the strategy of proof from~\cite{GrJo10} cannot %, to the best of our knowledge, 
%be (easily) carried over to multivariate situations. %does not seem to carry over to the multivariate case. 
As in~\cite{Ro}, by $\Phi'_-(u)$ and $\Phi'_+(u)$ we mean the left or, respectively, the right derivative of the function $\Phi$ at $u \in \R$.

\begin{thm} \label{thm:char-thm}
Suppose $f \in L^2(Q,\R)$ and $w\in S^{\infty}(Q,(0,\infty))$ and write $f^* := p_{\sigma}^w(f)$. Suppose further that $\Phi: I \to \R$ is a convex function on an interval $I \subset \R$ and let $\phi: I \to\R \cup \{\pm\infty\}$ be any function with $\Phi_-'(u) \le \phi(u) \le \Phi_+'(u)$ for $u \in I$,  
%let $\phi: I \to\R \cup \{\pm\infty\}$ be any subgradient function of $\Phi$, that is, any monotonically increasing function with $\phi(u) \in \partial \Phi(u)$  for $u \in I^{\circ}$ 
such that 
\begin{align} \label{eq:char-thm-ass-1}
I \supset f(Q)
\end{align} 
(or, more precisely, $I \supset f_0(Q)$ for some representative $f_0$ of $f$) and such that
\begin{align} \label{eq:char-thm-ass-2}
\Phi \circ f, \Phi \circ g \in L^1(Q,\R) \qquad \text{and} \qquad \phi\circ g \in L^2(Q,\R)
\end{align}
for all $g \in L^2_{\sigma}(Q,I)$. Then $f^*$ is a minimizer of the %generalized monotonic regression 
functional $J^{\Phi}_{f,w}|_{L^2_{\sigma}(Q,I)}$ with
\begin{align}
J^{\Phi}_{f,w}(g) := \int_Q \Delta_{\Phi}(f(x),g(x)) w(x) \d x 
\qquad (g \in L^2_{\sigma}(Q,I)),
\end{align}
where $\Delta_{\Phi}(u,v) := \Phi(u)-\Phi(v)-\phi(v)(u-v)$ for $u,v \in I$. If $\Phi$ is even strictly convex, then $f^*$ is the only minimizer of $J^{\Phi}_{f,w}|_{L^2_{\sigma}(Q,I)}$. 
\end{thm}

\begin{proof}
With our preparations from the above lemmas at hand, 
we can proceed along the lines of proof of the discrete version of the theorem (Theorem~1.5.1 in~\cite{RoWrDy}). 
Indeed,
\begin{align} \label{eq:char-thm-1}
\Delta_{\Phi}(r,t) = \Delta_{\Phi}(r,s) + \Delta_{\Phi}(s,t) + (r-s)(\phi(s)-\phi(t)) \qquad (r,s,t \in I) 
\end{align}
by straightforward calculation using the definition of $\Delta_{\Phi}$ and, moreover, 
\begin{align} \label{eq:char-thm-2}
f^* = p_{\sigma}^w(f) \in L^2_{\sigma}(Q,I)
\end{align}
by the assumption~\eqref{eq:char-thm-ass-1} and Corollary~\ref{cor:Av_f,w(f^*=c)=c}. Consequently, 
\begin{align} \label{eq:char-thm-3}
&\int_Q \Delta_{\Phi}(f(x),g(x)) w(x)\d x - \int_Q \Delta_{\Phi}(f(x),f^*(x)) w(x)\d x - \int_Q \Delta_{\Phi}(f^*(x),g(x)) w(x)\d x \notag \\
&\qquad \qquad = \int_Q \big(f(x)-f^*(x)\big)\big(\phi(f^*(x))-\phi(g(x))\big) w(x) \d x \notag \\
&\qquad \qquad = \scprd{f-f^*,\phi\circ f^*}_{2,w} - \scprd{f-f^*,\phi\circ g}_{2,w}
\qquad (g \in L^2_{\sigma}(Q,I)),
\end{align}
where all integrals are well-defined and finite by~\eqref{eq:char-thm-2} and our assumption~\eqref{eq:char-thm-ass-2}. 
Since $\Phi$ is convex by assumption, the function $\phi: I \to \R\cup\{\pm\infty\}$ is monotonically increasing and finite on the interior of $I$ by Theorem~24.1 of~\cite{Ro} (or more precisely the first part of it, for which no closedness assumption has to be imposed on $\Phi$). %along with the remarks preceding Theorem~24.2 of~\cite{Ro}) 
And therefore, by Lemma~\ref{lm:appr-by-bounded-mon-phi_n} and~(\ref{eq:char-thm-ass-2}.b), we can find bounded monotonically increasing functions $\phi_n: \R \to \R$ such that
\begin{align} \label{eq:char-thm-4}
\phi_n \circ g \underset{\norm{\cdot}_2}{\longrightarrow} \phi\circ g \qquad (n\to\infty) \qquad (g \in L^2_{\sigma}(Q,I)).
\end{align} 
Since $\phi_n$ is increasing for every $n\in \N$, we also have
\begin{align} \label{eq:char-thm-5}
\phi_n \circ g \in L^2_{\sigma}(Q,\R) \qquad (g \in L^2_{\sigma}(Q,I)).
\end{align}
Applying now~\eqref{eq:char-thm-4} and~\eqref{eq:char-thm-2} to~\eqref{eq:char-thm-3} and then using Lemma~\ref{lm:phi(f^*)} as well as  Proposition~\ref{prop:geometr-char-of-p} with~\eqref{eq:char-thm-5}, we obtain
\begin{align} \label{eq:char-thm-6}
\int_Q \Delta_{\Phi}(f(x),g(x)) w(x)\d x 
&\ge  \int_Q \Delta_{\Phi}(f(x),f^*(x)) w(x)\d x \notag \\
&\qquad + \int_Q \Delta_{\Phi}(f^*(x),g(x)) w(x)\d x
\qquad (g \in L^2_{\sigma}(Q,I)).
\end{align} 
Since $\phi(v)$ is a subgradient of $\Phi$ at $v$ for every $v \in I$ (Theorem~23.2 of~\cite{Ro}), we have 
\begin{align} \label{eq:char-thm-7}
\Delta_{\Phi}(u,v) \ge 0 \qquad (u,v\in I)
\end{align}
by the subradient inequality for convex functions. In view of~\eqref{eq:char-thm-2} and~\eqref{eq:char-thm-6}, \eqref{eq:char-thm-7}, it is now clear that $f^*$ is a minimizer of $J^{\Phi}_{f,w}|_{L^2_{\sigma}(Q,I)}$, as desired. 
In the special case where $\Phi$ is even strictly convex, we have strict inequality in~\eqref{eq:char-thm-7} for all $u, v \in I$ with $u\ne v$ and hence the second integral on the right-hand side of~\eqref{eq:char-thm-6} is strictly positive for every $g \in L^2_{\sigma}(Q,I)$ with $g \ne f^*$. So, $J^{\Phi}_{f,w}|_{L^2_{\sigma}(Q,I)}$ can have no other minimizer apart from $f^*$, as desired. 
\end{proof}

In the special case where $\Phi: I \to \R$ is a continuously differentiable convex function on a compact interval $I \supset f(Q)$, the integrability assumptions~\eqref{eq:char-thm-ass-2} are, of course,  automatically satisfied. 
In the extreme special case 
\begin{align}
\Phi(u) := u^2 \qquad (u \in I) \qquad \text{with} \qquad I := \R,
\end{align}
one has $\Delta_{\Phi}(u,v) = |u-v|^2$ for $u,v \in I$ and therefore the problem of minimizing $J_{f,w}^{\Phi}|_{L^2_{\sigma}(Q,I)}$ in this extreme special case is nothing but %is the same as
the standard monotonic regression problem of minimizing $J_{f,w}|_{L^2_{\sigma}(Q,\R)}$. In this sense, the problem of minimizing $J_{f,w}^{\Phi}|_{L^2_{\sigma}(Q,I)}$ for general convex functions $\Phi$ on general intervals $I$ is a generalized monotonic regression problem and the above theorem says that for given $f$ and $w$, all generalized regression problems with strictly convex functions $\Phi$ have the same solution %, namely $p_{\sigma}^w(f)$, 
as the standard monotonic regression problem, namely $p_{\sigma}^w(f)$.

\section{Continuity of monotonic regression functions}%{Continuity of monotonic regressors}

In this section, we deal with the special case of continuous functions $f$ and $w$. We will show that in this case the monotonic regression function $p_{\sigma}^w(f)$ is continuous as well and has a closed-form representation in terms of averaging expressions, namely
\begin{align}
a_{\sigma, \mathrm{is}}^w(f)(x) &:= \inf_{L \in \mathcal{L}^{\mathrm{ro}}_{\sigma}(x)} \sup_{U \in \mathcal{U}^{\mathrm{ro}}_{\sigma}(x)} \Av_{f,w}(L\cap U) \label{eq:a_is}\\
a_{\sigma, \mathrm{si}}^w(f)(x) &:=  \sup_{U \in \mathcal{U}^{\mathrm{ro}}_{\sigma}(x)} \inf_{L \in \mathcal{L}^{\mathrm{ro}}_{\sigma}(x)} \Av_{f,w}(L\cap U), 
\label{eq:a_si}
\end{align}
where $\mathcal{L}^{\mathrm{ro}}_{\sigma}(x) := \{ L \in \mathcal{L}_{\sigma}: L \text{ is relatively open in } Q \text{ and } L \ni x \}$ and $\mathcal{U}^{\mathrm{ro}}_{\sigma}(x) := \{ U \in \mathcal{U}_{\sigma}: U \text{ is relatively open in } Q \text{ and } U \ni x \}$ for $x \in Q$. Clearly, $L\cap U$ is non-null for every $L \in \mathcal{L}_{\sigma}^{\mathrm{ro}}(x)$ and $U \in \mathcal{U}_{\sigma}^{\mathrm{ro}}(x)$ and therefore~\eqref{eq:a_is} and~\eqref{eq:a_si} define  well-defined bounded functions for all $f \in L^{\infty}(Q,\R)$ and $w \in L^{\infty}(Q,[c,\infty))$ with $c \in (0,\infty)$ (Lemma~\ref{lm:Av_f,w-bounded-and-equicont}). 
\smallskip 

We begin by establishing an averaging formula in the case of grid-constant functions $f$ and $w$. Compared to the representation from Corollary~\ref{cor:avg-G}, the essential difference is that the sets the infimum and supremum are taken over do not depend on the grid $G$ %on which $f$ and $w$ are constant. 
on the cells of which $f$ and $w$ are constant. 

\begin{lm} \label{lm:avg-ro}
Suppose $G$ is an equidistant grid on $Q$ and $f:Q \to \R$ and $w:Q \to (0,\infty)$ are $G$-constant functions. Suppose further $\sigma \in \{-1,1\}^d$. Then $a_{\sigma, \mathrm{is}}^w(f)$ and $a_{\sigma, \mathrm{si}}^w(f)$ are $\sigma$-monotonic representatives of $p_{\sigma}^w(f)$ that are constant on the relative interior of every cell of $G$.
\end{lm}

\begin{proof}
We proceed in four steps, starting with some preliminary considerations. Set
\begin{align} \label{eq:avg-ro-1}
f_0^* := \sum_{x\in G} p_{\sigma}^{w|_G}(f|_G)(x) \chi_{C^G_{\sigma}(x)},
\end{align}
where $C^G_{\sigma}(x)$ for a given $x \in Q$ is the $\sigma$-lower semiclosed and $\sigma$-upper relatively semiopen cell of $G$ that contains $x$. What we mean by such a cell is a set $C$ of the form
\begin{align*}
C = I_1 \times \dotsb \times I_d
\end{align*} 
where the $I_i$ are partition intervals of the partitions $P_i = \{t_{ik}: k \in \{0, \dots, m_i\} \}$ defining $G$ of the following forms: in case $\sigma_i = 1$, 
\begin{align*}
I_i = [t_{i k-1},t_{i k}) \text{ for some } k \in \{1,\dots,m_i-1\} \text{ or } I_i = [t_{i m_i-1},t_{i m_i}]
\end{align*}
and, in case $\sigma_i = -1$, 
\begin{align*}
I_i = [t_{i 0},t_{i 1}] \text{ or } I_i = (t_{i k-1},t_{i k}] \text{ for some } k \in \{2,\dots,m_i\}.
\end{align*}
Since $\sigma \in \{-1,1\}^d$ by assumption, this is a complete case distinction. %covers all possible cases.  
\smallskip

As a first step, we show that $f_0^*$ is a $\sigma$-monotonic representative of $p_{\sigma}^w(f)$ that is constant on the relative interior of every $G$-cell. 
Indeed, by our definition of grids, every grid point $x \in G$ is the midpoint of its lower semiclosed and upper relatively semiopen cell $C^G(x)$ as well as of its $\sigma$-lower semiclosed and $\sigma$-upper relatively semiopen cell $C^G_{\sigma}(x)$. And therefore, $C^G_{\sigma}(x)$ and $C^G(x)$ can differ at most at the boundary, more precisely:
\begin{align} \label{eq:avg-ro-2}
\rint C^G_{\sigma}(x) = \rint C^G(x). 
\end{align}
So, by Theorem~\ref{thm:spec-case-grid-const-fct} and~\eqref{eq:avg-ro-2}, the function $f^*_0$ is a representative of $p_{\sigma}^w(f)$ and is constant on the relative interior of every $G$-cell. Also, $f^*_0$ is $\sigma$-monotonic by the $\sigma$-monotonicity of $p_{\sigma}^{w|_G}(f|_G)$ and the analog of Lemma~\ref{lm:cells}(i) for  $C^G_{\sigma}(x)$ instead of $C^G(x)$. 
\smallskip

As a second step, we show that $a_{\sigma,\mathrm{is}}^w(f)(x) = f_0^*(x)$ for every $x \in Q$.
So, let $x \in Q$ be fixed for the rest of this step and write $c := f_0^*(x)$. We then have for every $U \in \mathcal{U}_{\sigma}^{\mathrm{ro}}(x)$ and $L \in \mathcal{L}_{\sigma}^{\mathrm{ro}}(x)$ positive numbers $\eps_U, \eps_L >0$ such that
\begin{align*}
\{f_0^* \le c\} \cap U \supset C^G_{\sigma}(x) \cap B_{\eps_U}(x)
\qquad \text{and} \qquad
L \cap \{f_0^* \ge c\} \supset B_{\eps_L}(x) \cap C^G_{\sigma}(x) 
\end{align*}
and the sets on the right-hand sides of these inclusions clearly have a non-empty interior and thus are non-null sets. Consequently, 
\begin{align} 
\Av_{f,w}(\{f_0^* \le c\} \cap U) \le c \qquad (U \in \mathcal{U}_{\sigma}^{\mathrm{ro}}(x)) \label{eq:avg-ro-3} \\
\Av_{f,w}(L \cap \{f_0^* \ge c\}) \ge c \qquad (L \in \mathcal{L}_{\sigma}^{\mathrm{ro}}(x)) \label{eq:avg-ro-4}
\end{align}
by virtue of Corollary~\ref{cor:Av_f,w(f^*=c)=c}. Since $f_0^*$ is $\sigma$-monotonic by the first step, we have
\begin{align}
x \in \{f_0^* \le c\} \in \mathcal{L}_{\sigma} 
\qquad \text{and} \qquad 
x \in \{f_0^* \ge c\} \in \mathcal{U}_{\sigma} 
\end{align}
(Lemma~\ref{lm:char-isotonic-fcts}). 
Since, moreover, by~\eqref{eq:avg-ro-1} the sets $\{f_0^* \le c\}$ and $\{f_0^* \ge c\}$ are unions of $\sigma$-lower semiclosed $\sigma$-upper relatively semiopen grid cells, $\{f_0^* \le c\}$ is actually a relatively open $\sigma$-lower set in $Q$ and there exist relatively open $\sigma$-upper sets $U_n$ in $Q$ such that
\begin{align} \label{eq:avg-ro-5}
U_{n+1} \subset U_n \qquad (n \in \N) \qquad \text{and} \qquad \{f_0^* \ge c\} \subset \bigcap_{n=1}^{\infty} U_n \subset \ol{\{f_0^* \ge c\}}.
\end{align}
(Choose, for instance, $U_n := \{ y \in Q: |y_i-u_i|_{\infty} < 1/n  \text{ for some } u \in \{f_0^* \ge c\} \}$ for $n \in \N$.) So, we have
\begin{gather} 
\{f_0^* \le c\} \in \mathcal{L}_{\sigma}^{\mathrm{ro}}(x) \label{eq:avg-ro-6}\\
U_n \in \mathcal{U}_{\sigma}^{\mathrm{ro}}(x)
\qquad \text{and} \qquad
\Av_{f,w}(L \cap \{f_0^* \ge c\}) = \lim_{n\to\infty} \Av_{f,w}(L \cap U_n) \label{eq:avg-ro-7}
\end{gather}
for every $L \in \mathcal{L}_{\sigma}^{\mathrm{ro}}(x)$. And therefore
\begin{align} \label{eq:avg-ro-8}
a_{\sigma,\mathrm{is}}^w(f)(x) \le \sup_{U\in \mathcal{U}_{\sigma}^{\mathrm{ro}}(x)} \Av_{f,w}(\{f_0^* \le c\} \cap U) \le c
\end{align}
by virtue of~\eqref{eq:avg-ro-3} and~\eqref{eq:avg-ro-6}, as well as
\begin{align} \label{eq:avg-ro-9}
c \le \inf_{L \in \mathcal{L}_{\sigma}^{\mathrm{ro}}(x)} \Av_{f,w}(L \cap \{f_0^* \ge c\}) = \inf_{L \in \mathcal{L}_{\sigma}^{\mathrm{ro}}(x)} \lim_{n\to\infty} \Av_{f,w}(L \cap U_n) \le a_{\sigma,\mathrm{is}}^w(f)(x)
\end{align}
by virtue of~\eqref{eq:avg-ro-4} and~\eqref{eq:avg-ro-7}. Combining now~\eqref{eq:avg-ro-8} and~\eqref{eq:avg-ro-9}, we arrive at the desired conclusion $a_{\sigma,\mathrm{is}}^w(f)(x) = c = f_0^*(x)$ of the second step. 
\smallskip

As a third step, we prove the $\inf$-$\sup$ part of the lemma, which concerns $a_{\sigma,\mathrm{is}}^w(f)$. %of the lemma's assertion concerning $a_{\sigma,\mathrm{is}}^w(f)$. 
Indeed, by the first and second step, we see that $a_{\sigma,\mathrm{is}}^w(f) = f_0^*$ is a $\sigma$-monotonic representative of $p_{\sigma}^w(f)$ that is constant on the relative interior of every $G$-cell, which is the desired conclusion of the third step. %as desired.
\smallskip

As a fourth step, we prove the $\sup$-$\inf$ part of the lemma, which concerns $a_{\sigma,\mathrm{si}}^w(f)$. %of the lemma's assertion concerning $a_{\sigma,\mathrm{si}}^w(f)$. 
Indeed, as $\mathcal{U}_{\sigma}^{\mathrm{ro}}(x) = \mathcal{L}_{-\sigma}^{\mathrm{ro}}(x)$ and $\mathcal{L}_{\sigma}^{\mathrm{ro}}(x) = \mathcal{U}_{-\sigma}^{\mathrm{ro}}(x)$ by Lemma~\ref{lm:minus-sigma}(ii), we see that
\begin{align} \label{eq:avg-ro-10}
a_{\sigma,\mathrm{si}}^w(f)(x) 
&= -\inf_{U \in \mathcal{U}_{\sigma}^{\mathrm{ro}}(x)} \Big( -\inf_{L \in \mathcal{L}_{\sigma}^{\mathrm{ro}}(x)} \Av_{f,w}(L\cap U) \Big) \notag \\
&= - \inf_{U \in \mathcal{U}_{\sigma}^{\mathrm{ro}}(x)} \sup_{L \in \mathcal{L}_{\sigma}^{\mathrm{ro}}(x)} \big( -\Av_{f,w}(L\cap U) \big) \notag \\
&= - \inf_{U \in \mathcal{L}_{-\sigma}^{\mathrm{ro}}(x)} \sup_{L \in \mathcal{U}_{-\sigma}^{\mathrm{ro}}(x)} \Av_{-f,w}(L\cap U) 
= -a_{-\sigma,\mathrm{is}}^w(-f)(x) \qquad (x \in Q).
\end{align}
Applying the $\inf$-$\sup$ part of the lemma with the $G$-constant functions $-f, w$ and with $-\sigma \in \{-1,1\}^d$, we further see that $-a_{-\sigma,\mathrm{is}}^w(-f)$ is $\sigma$-monotonic representative of
\begin{align}
-p_{-\sigma}^w(-f) = p_{\sigma}^w(f)
\end{align} 
%$-p_{-\sigma}^w(-f) = p_{\sigma}^w(f)$ 
(Lemma~\ref{lm:minus-sigma}(i)) that is constant on the relative interior of every $G$-cell. In view of~\eqref{eq:avg-ro-10}, this is the desired conclusion of the fourth step.  %Combining this with~\eqref{eq:avg-ro-10}, we arrive at the desired conclusion of the fourth step.  
\end{proof}

\begin{lm} \label{lm:grids-with-x_0-in-interior-of-its-cell}
Suppose $x_0 \in Q$. Then there exist equidistant grids $G_n$ on $Q$ such that their cells' maximal edge length $\operatorname{len}_{G_n}$ tends to zero as $n \to \infty$ %$\operatorname{len}_{G_n} \longrightarrow 0$ as $n \to \infty$ 
and such that $x_0$ lies in the relative interior of its $G_n$-cell %$C^{G_n}(x_0)$ 
for every $n \in \N$. 
%In short, 
%\begin{align}
%\operatorname{len}_{G_n} \longrightarrow 0 \qquad (n\to\infty)
%\qquad \text{and} \qquad
%x_0 \in \rint C^{G_n}(x_0) \qquad (n\in\N).
%\end{align}
%
\end{lm}

\begin{proof}
We have to show that for every $\eps > 0$ there exists an equidistant grid $G$ on $Q$ such that
\begin{align} \label{eq:grids-with-x_0-in-interior-of-its-cell-1}
\operatorname{len}_G \le \eps 
\qquad \text{and} \qquad
x_0 \in \rint C^G(x_0).
\end{align}
So, let $\eps > 0$ and let $P := P_1 \times \dotsb \times P_d$ with partitions
\begin{align}
P_i := \big\{ a_i + \frac{k}{p_i}(b_i-a_i): i \in \{0,\dots,p_i\} \big\},
\end{align}
%of $[a_i,b_i]$, 
where the numbers $p_i$ are chosen as follows: in case $(x_{0i}-a_i)/(b_i-a_i) \notin \Q$,
\begin{align} \label{eq:grids-with-x_0-in-interior-of-its-cell-2}
p_i \in \N \qquad \text{and} \qquad p_i \ge 1/\eps
\end{align}
and in case $(x_{0i}-a_i)/(b_i-a_i) \in \Q$,
\begin{align} \label{eq:grids-with-x_0-in-interior-of-its-cell-3}
p_i \in \mathbb{P} \qquad \text{and} \qquad p_i \ge 1/\eps \qquad \text{and} \qquad p_i \ge n_i + 1
\end{align} 
with $\mathbb{P}$ being the set of all primes and with $n_i$ being the denominator of the irreducible fractional representation of $(x_{0i}-a_i)/(b_i-a_i) \in \Q$. Also, let $G$ be the grid on $Q$ that is determined by $P = P_1 \times \dotsb \times P_d$. 
Since the partitions $P_i$ are equidistant with $\operatorname{len}_{P_i} = 1/p_i \le \eps$ by virtue of~(\ref{eq:grids-with-x_0-in-interior-of-its-cell-2}.b) and (\ref{eq:grids-with-x_0-in-interior-of-its-cell-3}.b), our grid $G$ is equidistant and satisfies~(\ref{eq:grids-with-x_0-in-interior-of-its-cell-1}.a). 
%
%It remains to prove~(\ref{eq:grids-with-x_0-in-interior-of-its-cell-1}.b).
In order to prove~(\ref{eq:grids-with-x_0-in-interior-of-its-cell-1}.b), we now show that $x_0$ lies on none of the grid hyperplanes $H_{ik}$ that run through the interior of $Q$, that is,
\begin{align} \label{eq:grids-with-x_0-in-interior-of-its-cell-4}
x_0 \notin H_{ik} := \big\{x \in \R^d: x_i = a_i + \frac{k}{p_i}(b_i-a_i) \big\}
\end{align}
for every $k \in \{1,\dots,p_i-1\}$ and $i \in \{1,\dots,d\}$. Assume, on the contrary, that $x_{0i} = a_i + (k/p_i) (b_i-a_i)$ for some $k \in \{1,\dots,p_i-1\}$ and some  $i \in \{1,\dots,d\}$. Then 
\begin{align} \label{eq:grids-with-x_0-in-interior-of-its-cell-5}
\Q \cap (0,1) \ni \frac{k}{p_i} = \frac{x_{0i}-a_i}{b_i-a_i} = \frac{m_i}{n_i} 
\end{align}
for some coprime positive integers $m_i, n_i \in \N$. Consequently, 
\begin{align} \label{eq:grids-with-x_0-in-interior-of-its-cell-6}
p_i \in \mathbb{P} \qquad \text{and} \qquad p_i \ge n_i+1
\end{align}
by our choice of the $p_j$ and, moreover, $n_i$ divides $k n_i = m_i p_i$. As $m_i$ and $n_i$ are coprime, $n_i$ must divide $p_i$ and therefore $n_i = 1$ by virtue of~\eqref{eq:grids-with-x_0-in-interior-of-its-cell-6}. In view of~\eqref{eq:grids-with-x_0-in-interior-of-its-cell-5}, we thus obtain the desired contradiction
\begin{align}
1 > \frac{k}{p_i} = \frac{m_i}{n_i} = m_i \ge 1.
\end{align}
So, \eqref{eq:grids-with-x_0-in-interior-of-its-cell-4} is proven and this in turn implies~(\ref{eq:grids-with-x_0-in-interior-of-its-cell-1}.b).
\end{proof}

With the above lemmas at hand, we can now establish the main result of the section. It is a generalization of a result from~\cite{GrJo10} (Theorem~1 and Lemma~1), where the case of univariate functions $f$ and $w$ is considered. We proceed in a completely different way than~\cite{GrJo10} because the strategy of proof from~\cite{GrJo10} does not carry over to multivariate situations. %cannot %, to the best of our knowledge, 
%be (easily) carried over to multivariate situations. 
%See the remarks below for an alternative proof of the continuity result in the univariate special case from~\cite{GrJo10}. 
After the proof below, we also sketch a simple alternative proof of the continuity result in the univariate special case from~\cite{GrJo10}. 

\begin{thm} \label{thm:continuity-thm}
Suppose $f \in C(Q,\R)$ and $w \in C(Q,(0,\infty))$. Then $p_{\sigma}^w(f)$ has a unique continuous and $\sigma$-monotonic representative $f_0^*$ and %it is given by
\begin{align} \label{eq:cont-thm}
a_{\hat{\sigma},\mathrm{is}}^{w(\cdot\,\&\dot{x})}\big(f(\cdot\,\&\dot{x})\big)(\hat{x}) = f_0^*(x) = a_{\hat{\sigma},\mathrm{si}}^{w(\cdot\,\&\dot{x})}\big(f(\cdot\,\&\dot{x})\big)(\hat{x})
\qquad (x = \hat{x} \& \dot{x} \in Q).
\end{align}
\end{thm}

\begin{proof}
(i) We first confine ourselves to the special case $\sigma \in \{-1,1\}^d$ and prove in two steps that 
\begin{align} \label{eq:cont-thm-1}
f_{\#} := a_{\sigma, \#}^w(f) \qquad (\# \in \{\mathrm{is},\mathrm{si}\})
\end{align}
is a $\sigma$-monotonic representative of $p_{\sigma}^w(f)$ that is  continuous at every $x_0 \in Q$.
%We first prove the assertion in the special case $\sigma \in \{-1,1\}^d$ in three steps. We begin by showing that
%\begin{align} \label{eq:cont-thm-1}
%f_{\#} := a_{\sigma, \#}^w(f) \qquad (\# \in \{\mathrm{is},\mathrm{si}\})
%\end{align}
%is a $\sigma$-monotonic representative of $p_{\sigma}^w(f)$ being continuous at every $x_0 \in Q$ and that $f_{\mathrm{is}} = f_{\mathrm{si}}$.
%
So, let $x_0 \in Q$ be fixed for the rest of part~(i) of the proof and let $G_n$ be equidistant grids on $Q$ such that
\begin{align} \label{eq:cont-thm-2}
\operatorname{len}_{G_n} \longrightarrow 0 \qquad (n\to\infty)
\qquad \text{and} \qquad
x_0 \in \rint C^{G_n}(x_0) \qquad (n\in\N)
\end{align} 
(Lemma~\ref{lm:grids-with-x_0-in-interior-of-its-cell}). Also, define the $G_n$-constant functions
\begin{align} \label{eq:cont-thm-3}
f_n := \sum_{x\in G_n} f(x) \chi_{C^{G_n}(x)} 
\qquad \text{and} \qquad
w_n := \sum_{x\in G_n} w(x) \chi_{C^{G_n}(x)}
\end{align}
as well as the functions $f_{n\#} := a_{\sigma, \#}^{w_n}(f_n)$ for $\# \in \{\mathrm{is},\mathrm{si}\}$ and $n \in \N$. (It should be noticed that while the grids $G_n$ -- and hence also the functions $f_n, w_n, f_{n\#}$ -- depend on $x_0$, the function $f_{\#}$ does not.)
\smallskip

As a first step, we show that $f_{\#}$ is a $\sigma$-monotonic representative of $p_{\sigma}^w(f)$ for $\# \in \{\mathrm{is},\mathrm{si}\}$. 
Since, by assumption, $f$ and $w$ are continuous and hence uniformly continuous on $Q$, it follows by~(\ref{eq:cont-thm-2}.a) and~\eqref{eq:cont-thm-3} that
\begin{gather} 
f_n \underset{\norm{\cdot}_{\infty}}{\longrightarrow} f \qquad (n\to\infty) \qquad \text{and} \qquad w_n \underset{\norm{\cdot}_{\infty}}{\longrightarrow} w \qquad (n\to\infty) \label{eq:cont-thm-4} \\
\ul{c} \le w_n(x) \le \ol{c} \qquad (x \in Q \text{ and } n \in \N) \label{eq:cont-thm-5} 
\end{gather}
for some positive constants $\ul{c},\ol{c} \in (0,\infty)$. So, by Lemma~\ref{lm:Av_f,w-bounded-and-equicont},
\begin{align} \label{eq:cont-thm-6} 
\sup_{E \in \mathcal{Q}_Q^{>}} \big| \Av_{f_n,w_n}(E) - \Av_{f,w}(E) \big| \longrightarrow 0 \qquad (n\to\infty)
\end{align} 
and $\mathcal{Q}_Q^{>} \ni E \mapsto \Av_{f_n,w_n}(E), \Av_{f,w}(E)$ are bounded functions. So, by Lemma~\ref{lm:difference-of-inf-sup-expressions},
\begin{align} \label{eq:cont-thm-7} 
\sup_{x\in Q} \big| f_{n\#}(x)-f_{\#}(x) \big| 
&\le \sup_{x\in Q} \sup_{(L,U) \in \mathcal{L}_{\sigma}^{\mathrm{ro}}(x) \times \mathcal{U}_{\sigma}^{\mathrm{ro}}(x)} \Big| \Av_{f_n,w_n}(L\cap U) - \Av_{f,w}(L\cap U) \Big| \notag \\
&\le \sup_{E \in \mathcal{Q}_Q^{>}} \big| \Av_{f_n,w_n}(E) - \Av_{f,w}(E) \big| \qquad (n \in \N).
\end{align}
%for all $n \in \N$. 
Combining~\eqref{eq:cont-thm-6} and~\eqref{eq:cont-thm-7}, we obtain the uniform convergence
\begin{align} \label{eq:cont-thm-8} 
f_{n \#} \underset{\norm{\cdot}_{\sup}}{\longrightarrow} f_{\#} \qquad (n\to\infty).
\end{align}
Since $f_{n\#} = a_{\sigma, \#}^{w_n}(f_n)$ is a representative of $p_{\sigma}^{w_n}(f_n)$ by Lemma~\ref{lm:avg-ro}, we also have that
\begin{align} \label{eq:cont-thm-9} 
f_{n\#} = p_{\sigma}^{w_n}(f_n) \underset{\norm{\cdot}_2}{\longrightarrow} p_{\sigma}^w(f) \qquad (n\to\infty)
\end{align}
by virtue of~\eqref{eq:cont-thm-4} and~\eqref{eq:cont-thm-5} and Lemma~\ref{lm:p^w(f)-contin-in-w,f}. In view of~\eqref{eq:cont-thm-8} and~\eqref{eq:cont-thm-9}, $f_{\#}$ is a representative of $p_{\sigma}^w(f)$. Since, moreover, the functions $f_{n\#}$ are $\sigma$-monotonic by Lemma~\ref{lm:avg-ro}, the function $f_{\#}$, in view of~\eqref{eq:cont-thm-8}, is $\sigma$-monotonic as well. %by~\eqref{eq:cont-thm-8}. 
\smallskip

As a second step, we show that $f_{\#}$ is continuous at $x_0$ for $\# \in \{ \mathrm{is},\mathrm{si}\}$. 
So, let $\eps > 0$. In view of~\eqref{eq:cont-thm-8}, there is an $n_0 \in \N$ such that
\begin{align} \label{eq:cont-thm-10} 
\norm{f_{n_0 \#}-f_{\#}}_{\sup} \le \eps/2.
\end{align}
Also, in view of~(\ref{eq:cont-thm-2}.b), there is a $\delta > 0$ such that
\begin{align} \label{eq:cont-thm-11} 
B_{\delta}(x_0) \cap Q \subset \rint C^{G_{n_0}}(x_0).
\end{align}
and therefore, by the constancy of $f_{n_0\#}$  on the relative interior of the cells of $G_{n_0}$ shown in Lemma~\ref{lm:avg-ro}, 
\begin{align} \label{eq:cont-thm-12} 
f_{n_0\#}(x) = f_{n_0\#}(x_0) \qquad (x \in B_{\delta}(x_0) \cap Q).
\end{align}
Combining now~\eqref{eq:cont-thm-10} and~\eqref{eq:cont-thm-12}, we see that
\begin{align} \label{eq:cont-thm-13}
|f_{\#}(x)- f_{\#}(x_0)| \le \eps \qquad (x \in B_{\delta}(x_0) \cap Q).
\end{align}
In other words, $f_{\#}$ is continuous at $x_0$, as desired.
\smallskip

%As a third step, we show that $f_{\mathrm{is}} = f_{\mathrm{si}}$. 
%%
%Indeed, by the first and second step, $f_{\mathrm{is}}$ and $f_{\mathrm{si}}$ both are continuous representatives of the same equivalence class $p_{\sigma}^w(f)$ and therefore must coincide everywhere, as desired.
%\smallskip

(ii) We now move on to the general case $\sigma \in \{-1,0,1\}^d$ and prove the assertion of the theorem in three steps. 
We define the functions $f_{\#}$ for $\# \in \{ \mathrm{is},\mathrm{si}\}$ by
\begin{align} \label{eq:cont-thm-14}
f_{\#}(x) := a_{\hat{\sigma},\#}^{w(\cdot\,\&\dot{x})}\big(f(\cdot\,\&\dot{x})\big)(\hat{x}) \qquad (x = \hat{x}\&\dot{x} \in Q).
\end{align}

As a first step, we observe that $f_{\#}(\cdot\,\&\dot{x})$ for every $\dot{x} \in \dot{Q}$ is a continuous $\hat{\sigma}$-monotonic representative of $p_{\hat{\sigma}}^{w(\cdot\,\&\dot{x})}\big(f(\cdot\,\&\dot{x})\big)$. 
Indeed, this immediately follows from part~(i) of the proof applied to %$f(\cdot\,\&\dot{x}) \in C(\hat{Q},\R)$ and $w(\cdot\,\&\dot{x}) \in C(\hat{Q},(0,\infty))$. 
\begin{align*}
f(\cdot\,\&\dot{x}) \in C(\hat{Q},\R) \qquad \text{and} \qquad w(\cdot\,\&\dot{x}) \in C(\hat{Q},(0,\infty)).
\end{align*}

As a second step, we show that $f_{\#}$ is continuous.
So, let $x,x_n \in Q$ with $x_n \longrightarrow x$ as $n\to\infty$. %$x \in Q$ be fixed and $x_n \in Q$ with $x_n \longrightarrow x$ as $n\to\infty$. 
Since $f$ and $w$ are continuous at $x$, we have
\begin{align} \label{eq:cont-thm-16}
f(\cdot\,\&\dot{x}_n) \underset{\norm{\cdot}_{\infty}}{\longrightarrow} f(\cdot\,\&\dot{x})
\qquad \text{and} \qquad
w(\cdot\,\&\dot{x}_n) \underset{\norm{\cdot}_{\infty}}{\longrightarrow} w(\cdot\,\&\dot{x})
\end{align}  
and therefore, by virtue of Lemma~\ref{lm:difference-of-inf-sup-expressions} and Lemma~\ref{lm:Av_f,w-bounded-and-equicont},
\begin{align} \label{eq:cont-thm-17}
\big| f_{\#}(\hat{x}_n \& \dot{x}_n) - f_{\#}(\hat{x}_n \& \dot{x}) \big|
&\le \sup_{\hat{E} \in \mathcal{Q}_{\hat{Q}}^{>}} \Big| \Av_{f(\cdot\,\&\dot{x}_n), w(\cdot\,\&\dot{x}_n)}(\hat{E})-\Av_{f(\cdot\,\&\dot{x}), w(\cdot\,\&\dot{x})}(\hat{E}) \Big| \notag \\
&\longrightarrow 0 \qquad (n\to\infty).
\end{align}
Since, moreover, $f_{\#}(\cdot\,\&\dot{x})$ is continuous by the first step, we also have
\begin{align} \label{eq:cont-thm-18}
\big| f_{\#}(\hat{x}_n \& \dot{x}) - f_{\#}(\hat{x} \& \dot{x}) \big| 
\longrightarrow 0 \qquad (n \to\infty). 
\end{align}
In view of~\eqref{eq:cont-thm-17} and~\eqref{eq:cont-thm-18}, the asserted  continuity of $f_{\#}$ is now clear. %Combining now~\eqref{eq:cont-thm-17} and~\eqref{eq:cont-thm-18}, we obtain the desired continuity of $f_{\#}$.
\smallskip

As a third step, we show that $f_{\#}$ is a $\sigma$-monotonic representatitive of $p_{\sigma}^w(f)$ and conclude the assertion of the theorem.  
Indeed, $f_{\#}(\cdot\,\&\dot{x})$ is $\hat{\sigma}$-monotonic for every $\dot{x} \in \dot{Q}$ by the first step and therefore $f_{\#}$ is $\sigma$-monotonic by Lemma~\ref{lm:sigma-hat} and thus, taking into account the second step, 
\begin{align}
f_{\#} \in L_{\sigma}^2(Q,\R).
\end{align}
Since, moreover, for every $g \in S_{\sigma}(Q,\R)$ and $\dot{x} \in \dot{Q}$, the function $g(\cdot\,\&\dot{x})$ is $\hat{\sigma}$-monotonic (Lemma~\ref{lm:sigma-hat}) and $f_{\#}(\cdot\,\&\dot{x})$ is a representative of $p_{\hat{\sigma}}^{w(\cdot\,\&\dot{x})}\big(f(\cdot\,\&\dot{x})\big)$ by the first step, we have
\begin{gather}
\scprd{f-f_{\#},f_{\#}-g}_{2,w} = \int_{\dot{Q}} \int_{\hat{Q}} (f(x)-f_{\#}(x))(f_{\#}(x)-g(x)) w(x) \d\hat{x} \d\dot{x} \notag \\
= \int_{\dot{Q}} \big\langle f(\cdot\,\&\dot{x})-f_{\#}(\cdot\,\&\dot{x}),f_{\#}(\cdot\,\&\dot{x})-g(\cdot\,\&\dot{x}) \big\rangle_{2,w(\cdot\,\&\dot{x})} \d\dot{x}
\ge 0 \label{eq:cont-thm-15}
\end{gather}
by virtue of Proposition~\ref{prop:geometr-char-of-p}. As $S_{\sigma}(Q,\R)$ is dense in $L^2_{\sigma}(Q,\R)$ by Theorem~\ref{thm:approx-thm-2-norm}, the inequality~\eqref{eq:cont-thm-15} extends to arbitrary $g \in L^2_{\sigma}(Q,\R)$ and therefore $f_{\#}$ is a representative of $p_{\sigma}^w(f)$ by Proposition~\ref{prop:geometr-char-of-p}, as desired.
Summing up, we now know that $f_{\mathrm{is}}$ and $f_{\mathrm{si}}$ both are continuous $\sigma$-monotonic representatives of $p_{\sigma}^w(f)$ and, as any equivalence class can have at most one continuous representative, the assertion of the theorem follows. 
\end{proof}

In the univariate special case with $\sigma = 1$, we obviously have $\mathcal{L}_{\sigma}^{\mathrm{ro}}(x) = \{ [a,v): v \in (x,b]\} \cup \{ [a,b] \}$ and $\mathcal{U}_{\sigma}^{\mathrm{ro}}(x) = \{ (u,b]: u \in [a,x) \} \cup \{ [a,b] \}$ for every $x \in Q = [a,b]$ and therefore the general averaging formula~\eqref{eq:cont-thm} from the above theorem reduces to the well-known univariate formula
\begin{align} \label{eq:univariate-inf-sup-formula}
\inf_{v \in (x,b]} \sup_{u \in [a,x)} \bigg( \int_u^v fw \d\lambda \bigg) \bigg/ \bigg( \int_u^v w \d\lambda \bigg)
&= f^*_0(x)  \\
&= \sup_{u \in [a,x)} \inf_{v \in (x,b]} \bigg( \int_u^v fw \d\lambda \bigg) \bigg/ \bigg( \int_u^v w \d\lambda \bigg) \notag
\end{align} 
from the literature~\cite{Ma91} (Section~2), \cite{MaMaTuWa01} (Section~6.1) and~\cite{AnSo11} (Lemma~2). 
Incidentally, the formula~\eqref{eq:univariate-inf-sup-formula} also yields a simple alternative proof of the continuity of $f_0^*$ in the univariate special case. %We have only to realize that 
Indeed, by the continuity of $f$ and $w$, the map
\begin{align*}
[a,x) \times (x,b] \ni (u,v) \mapsto \bigg( \int_u^v fw \d\lambda \bigg) \bigg/ \bigg( \int_u^v w \d\lambda \bigg)
\end{align*}
for every $x \in (a,b)$ extends to a continuous -- hence uniformly continuous -- map $\phi: [a,b] \times [a,b] \to \R$ defined by
\begin{align}
\phi(u,u) := f(u) \qquad \text{and} \qquad \phi(u,v) := \bigg( \int_u^v fw \d\lambda \bigg) \bigg/ \bigg( \int_u^v w \d\lambda \bigg)
\end{align}
for $u \in [a,b]$ or $u,v \in [a,b]$ with $u\ne v$, respectively. And from this uniform continuity of $\phi$, in turn, the desired continuity of 
\begin{align}
[a,b] \ni x \mapsto \min_{v\in [x,b]} \max_{u \in [a,x]} \phi(u,v) = \inf_{v\in (x,b]} \sup_{u \in [a,x)} \phi(u,v) = f_0^*(x)
\end{align}
easily follows. 
We close this section with a few remarks on the regularity (degree of differentiability) of monotonic regression functions. While continuity, by the above continuity theorem, is preserved under the monotonic regression operator $p_{\sigma}^w$, higher degrees of regularity are in general not preserved (the method from~\cite{HaHu01}, by contrast, does not have that drawback). %under $p_{\sigma}^w$
Indeed, it is well-known already from the univariate case that the monotonic regression $p_{\sigma}^w(f)$ of a smooth function $f$ with $w \equiv 1$ will have kinks, in general. See, for instance, Figure~1 (lower row) from~\cite{LiDu14}.
We do have, however, that %At least, however, %Yet, 
monotonic regression functions $p_{\sigma}^w(f)$ can at least be approximated arbitrarily well by smooth $\sigma$-monotonic functions. This is because standard mollification essentially respects $\sigma$-monotonicity.  In the following, for a rectangular set $X = \bigtimes_{i=1}^d [a_i',b_i']$ with $a_i' < b_i'$, we use the abbreviations
\begin{align*}
C_{\sigma}^m(X,\R) := \{f \in C^m(X,\R): f \text{ is $\sigma$-monotonic} \} 
\quad \text{and} \quad
X^{-\delta} := \bigtimes_{i=1}^d [a_i'+\delta,b_i'-\delta]
\end{align*} 
for $m \in \N\cup\{0,\infty\}$ and $\delta \in \R$ such that $a_i'+\delta < b_i'-\delta$ for all $i\in\{1,\dots,d\}$.

\begin{cor} \label{cor:C_sigma^infty-dense-in-L^2_sigma}
\begin{itemize}
\item[(i)] If $f \in C(Q,\R)$ and $w \in C(Q,(0,\infty))$, then $p_{\sigma}^w(f)|_{Q^{-r}}$ for every $0 < r < \min_{i\in\{1,\dots,d\}} (b_i-a_i)/2$ can be approximated arbitrarily well w.r.t.~$\norm{\cdot}_{\infty}$ by functions from $C_{\sigma}^{\infty}(Q^{-r},\R)$. 
\item[(ii)] If $f \in L^2(Q,\R)$ and $w \in C(Q,(0,\infty))$, then $p_{\sigma}^w(f)$  can be approximated arbitrarily well w.r.t.~$\norm{\cdot}_{2}$ by functions from $C_{\sigma}^{\infty}(Q,\R)$. 
In particular, $C_{\sigma}^{\infty}(Q,\R)$ is dense in $L^2_{\sigma}(Q,\R)$ w.r.t.~$\norm{\cdot}_2$. 
\end{itemize}
\end{cor}

\begin{proof}
We begin with a preparatory step showing that standard mollification essentially respects $\sigma$-monotonicity: more precisely, we show that if $g \in C_{\sigma}(X,\R)$ on a rectangular domain $X = \bigtimes_{i=1}^d [a_i',b_i']$ with $a_i' < b_i'$, then 
\begin{align}
\big( j_{\delta} * g \big)|_{X^{-\delta}} \in C_{\sigma}^{\infty}(X^{-\delta},\R)
\end{align}
for every $0 < \delta < \min_{i\in\{1,\dots,d\}} (b_i'-a_i')/2$. 
As usual, the convolution $j_{\delta}*g$ is given by
\begin{align}
\big( j_{\delta} * g \big)(x) := \int_{\R^d} j_{\delta}(x-y) \tilde{g}(y) \d y 
&= \int_{\R^d} j_{\delta}(y) \tilde{g}(x-y) \d y \notag \\ 
&= \int_{B_{\delta}(0)} j_{\delta}(y) \tilde{g}(x-y) \d y 
\label{eq:convolution-def}
\end{align}
for every $x \in \R^d$, where $\tilde{g}$ is the zero extension of $g$ from its domain $X$ to the whole of $\R^d$ and $j_{\delta}$ is the mollifier defined by $j_{\delta}(x) := (1/\delta)^{d} j(x/\delta)$ with a function $j \in C^{\infty}(\R^d,\R)$ such that
\begin{align}
\supp j \subset B_1(0) \qquad \text{and} \qquad j \ge 0 \qquad \text{and} \qquad \int_{\R^d} j(x) \d x = 1.  
\end{align}
%As usual, 
%$j_{\delta}$ is the mollifier defined by $j_{\delta}(x) := (1/\delta)^{d} j(x/\delta)$ with a function $j \in C^{\infty}(\R^d,\R)$ such that
%\begin{align}
%j \ge 0 \qquad \text{and} \qquad \supp j \subset B_1(0) \qquad \text{and} \qquad \int_{\R^d} j(x) \d x = 1.  
%\end{align}
%Also, the convolution $(j_{\delta}*g)(x)$ is defined for every $x \in \R^d$ by
%\begin{align*}
%\big( j_{\delta} * g \big)(x) := \int_{\R^d} j_{\delta}(x-y) \tilde{g}(y) \d y = \int_{\R^d} j_{\delta}(y) \tilde{g}(x-y) \d y 
%= \int_{B_{\delta}(0)} j_{\delta}(y) \tilde{g}(x-y) \d y, 
%\end{align*}
%where $\tilde{g}$ is the zero extension of $g$ from $X$ to the whole of $\R^d$.
It is well-known that $j_{\delta}*g \in C^{\infty}(\R^d,\R)$ and it thus remains to show that $( j_{\delta} * g )|_{X^{-\delta}}$ is $\sigma$-monotonic. So, let $x,x' \in X^{-\delta}$ with $0 < \delta < \min_{i\in\{1,\dots,d\}} (b_i'-a_i')/2$. We then have, of course,
\begin{align}
x-y, x'-y \in X \qquad (y \in B_{\delta}(0)).
\end{align}
So, by the assumed $\sigma$-monotonicity of $g$ on $X$ and by $j_{\delta} \ge 0$, it follows that 
\begin{align*}
j_{\delta}(y) \tilde{g}(x-y) = j_{\delta}(y) g(x-y) \le j_{\delta}(y) g(x'-y) = j_{\delta}(y) \tilde{g}(x'-y)
\qquad (y \in B_{\delta}(0))
\end{align*}
and therefore $( j_{\delta} * g )(x) \le ( j_{\delta} * g )(x')$ by virtue of~\eqref{eq:convolution-def}, as desired. With these preparatory considerations, the assertions~(i) and (ii) are now easy to prove.
\smallskip

(i) Suppose $f \in C(Q,\R)$ and $w \in C(Q,(0,\infty))$, write $f^* := p_{\sigma}^w(f)$ and let $0 < r < \min_{i\in\{1,\dots,d\}} (b_i-a_i)/2$. In view of the continuity theorem (Theorem~\ref{thm:continuity-thm}), $f^*$ has a unique representative $f_0^* \in C_{\sigma}(Q,\R)$ and therefore, by our preparatory step,
\begin{align}
\big( j_{1/n} * f_0^*)|_{Q^{-1/n}} \in C_{\sigma}^{\infty}(Q^{-1/n},\R)
\end{align}
for all $n \in \N$ with $1/n < \min_{i\in\{1,\dots,d\}} (b_i-a_i)/2$. Since $Q^{-r} \subset Q^{-1/n}$ for $1/n \le r$, we have in particular
\begin{align} \label{eq:C_sigma^infty-dense-(i),1}
\big( j_{1/n} * f_0^*)|_{Q^{-r}} \in C_{\sigma}^{\infty}(Q^{-r},\R)
\end{align}
for all $n \in \N$ with $1/n \le r$. Since, moreover, $Q^{-r}$ is a compact subset of $\operatorname{int} Q$, we further have
\begin{align} \label{eq:C_sigma^infty-dense-(i),2}
\norm{ \big( j_{1/n} * f_0^*)|_{Q^{-r}} - f^*|_{Q^{-r}} }_{\infty} = \sup_{x \in Q^{-r}} \big| \big( j_{1/n} * f_0^*)(x) - f_0^*(x) \big| 
\longrightarrow 0 
%\qquad (n \to \infty).
\end{align}
as $n \to \infty$. In view of~\eqref{eq:C_sigma^infty-dense-(i),1} and~\eqref{eq:C_sigma^infty-dense-(i),2}, the assertion~(i) is now clear. 
\smallskip

(ii) Suppose $f \in L^2(Q,\R)$ and $w \in C(Q,(0,\infty))$, write $f^* := p_{\sigma}^w(f)$ and let $\eps >0$. Choose $f_0 \in C(Q,\R)$ such that
\begin{align} \label{eq:C_sigma^infty-dense-(ii),1}
(\ol{c}/\ul{c}) \norm{f_0-f}_2 \le \eps/3,
\end{align} 
where $\ul{c} := \min_{x\in Q}w(x)$ and $\ol{c} := \max_{x \in Q} w(x)$. 
In view of the continuity theorem (Theorem~\ref{thm:continuity-thm}), $p_{\sigma}^w(f_0)$ has a unique representative $f_0^* \in C_{\sigma}(Q,\R)$ and it is easily verified that the constant extension of $f_0^*$ from $Q$ to the whole of $\R^d$ is continuous and $\sigma$-monotonic as well. In short, $f_0^* \circ p_Q \in C_{\sigma}(\R^d,\R)$, where $p_Q: \R^d \to \Q$ is the projection onto $Q$ defined by $(p_Q(x))_i = x_i$ in case $x_i \in [a_i,b_i]$ and 
\begin{align*}
\big(p_Q(x)\big)_i := a_i \qquad (x_i \in (-\infty,a_i)) 
\qquad \text{and} \qquad 
\big(p_Q(x)\big)_i := b_i \qquad (x_i \in (b_i,\infty))
\end{align*}
%\begin{align}
%\big(p_Q(x)\big)_i := 
%\begin{cases}
%a_i, \quad x_i \in (-\infty,a_i), \\
%x_i, \quad x_i \in [a_i,b_i], \\
%b_i, \quad x_i \in (b_i,\infty)
%\end{cases}
%\end{align}  
for $i \in \{1,\dots,d\}$ and $x \in \R^d$. In particular, $g_n := (f_0^* \circ p_Q)|_{Q^{1/n}} \in C_{\sigma}(Q^{1/n},\R)$ and therefore, by our preparatory step, 
\begin{align} \label{eq:C_sigma^infty-dense-(ii),2}
\big( j_{1/n} * g_n \big)|_Q \in C_{\sigma}^{\infty}(Q,\R). 
\end{align} 
With the help of Young's inequality and the fact that the zero extensions $\tilde{g}_n, \tilde{f}_0^*$ of $g_n$ and $f_0^*$ beyond their respective domains $Q^{1/n}$ and $Q$ differ at most on $Q^{1/n} \setminus Q$, we find
\begin{align}
\norm{ \big( j_{1/n} * g_n \big)|_Q - \big( j_{1/n} * f_0^* \big)|_Q }_2
&\le \norm{ j_{1/n} * g_n  - j_{1/n} * f_0^* }_2 = \big\| j_{1/n} * (\tilde{g}_n - \tilde{f_0^*}) \big\|_2 \notag \\
&\le \| \tilde{g}_n - \tilde{f_0^*} \|_2
\le \norm{f_0^*}_{\infty} \cdot \lambda(Q^{1/n}\setminus Q)^{1/2}
\end{align}
for every $n \in \N$, and with the help of the $\norm{\cdot}_{2,w}$-contractivity of $p_{\sigma}^w$ (Proposition~\ref{prop:properties-of-p}) and~\eqref{eq:C_sigma^infty-dense-(ii),1}, we find
\begin{align}
\norm{f_0^* - f^*}_2 \le (1/\ul{c}) \norm{f_0^*-f^*}_{2,w} \le (1/\ul{c}) \norm{f_0-f}_{2,w} \le (\ol{c}/\ul{c}) \norm{f_0-f}_2 \le \eps/3
\end{align}
for every $n \in \N$. 
Consequently, for $n$ large enough, we have
\begin{align} \label{eq:C_sigma^infty-dense-(ii),3}
\norm{ \big( j_{1/n} * g_n \big)|_Q - f^*}_2 
&\le \norm{f_0^*}_{\infty} \cdot \lambda(Q^{1/n}\setminus Q)^{1/2} + \norm{\big( j_{1/n} * f_0^* \big)|_Q  - f_0^*}_2 + \eps/3 \notag \\
&\le \eps. 
\end{align}
In view of~\eqref{eq:C_sigma^infty-dense-(ii),2} and~\eqref{eq:C_sigma^infty-dense-(ii),3}, the assertion~(ii) is now clear. 
\end{proof}

\section*{Acknowledgment}

I would like to thank Martin von Kurnatowski and Jan Schwientek as well as J\"urgen Franke, Patrick Link, Anke Stoll, and Rebekka Zache for interesting and inspiring discussions about incorporating monotonicity knowledge into machine-learning models. %the monotonization of machine-learning models. %Additionally, I would like thank J\"urgen Franke for pointing out reference~\cite{LiDu14} to me. 

\begin{small}

\end{small}

\end{document}